\documentclass[bj,authoryear]{imsart}

\usepackage{algorithm, algorithmic}

\startlocaldefs
\theoremstyle{plain}

\newtheorem{theorem}{Theorem}
\newtheorem{corollary}{Corollary}[section]
\newtheorem{lemma}{Lemma}[section]
\theoremstyle{remark}

\newtheorem*{remark}{Remark}
\newtheorem{assumption}{Assumption}

\endlocaldefs

\begin{document}

\begin{frontmatter}
\title{Controlling \mbox{FSR} in Selective Classification}
\runtitle{Controlling \mbox{FSR} in Selective Classification}

\begin{aug}
\author[A]{\inits{Z.}\fnms{Guanlan}~\snm{Zhao}\ead[label=e1]{22135047@zju.edu.cn}}

\author[B]{\inits{S.}\fnms{Zhonggen}~\snm{Su}\ead[label=e2]{suzhonggen@zju.edu.cn}*}

\address[A]{School of Mathematical Sciences,
Zhejiang University, Hangzhou, China\printead[presep={,\ }]{e1}}

\address[B]{School of Mathematical Sciences,
Zhejiang University, Hangzhou, China\printead[presep={,\ }]{e2}}

\end{aug}

\begin{abstract}
Uncertainty quantification and false selection error rate (\mbox{FSR}) control are crucial in many high-consequence scenarios, so we need models with good interpretability. This article introduces the optimality function for the binary classification problem in selective classification. We prove the optimality of this function in oracle situations and provide a data-driven method under the condition of exchangeability. We demonstrate it can control global \mbox{FSR} with the finite sample assumption and successfully extend the above situation from binary to multi-class classification. Furthermore, we demonstrate that \mbox{FSR} can still be controlled without exchangeability, ultimately completing the proof using the martingale method.
\end{abstract}


\end{frontmatter}

\section{Introduction}

With the advent of the big data era, our data modeling techniques have evolved from traditional low-data models to encompass the novel, big-data-centric approach that can handle tremendous volumes of information – such as the deep learning model. While deep learning has recorded substantial advancements in fields such as Computer Vision (CV) and Natural Language Processing (NLP), it has also spotlighted its inherent limitation, which hinges on the lack of clarity regarding our faith in the model's outputs, its only point of reference tends to be the performance of past data. Normally, it is assumed that the training and testing data originate from the same distribution, but this assumption does not often hold in real-life scenarios.

Uncertainty quantification and error control are crucial in many high-consequence scenarios, such as the medical and financial fields. We need models with stellar interpretability since incorrect outcomes can often lead to severe consequences. Therefore, to balance the processing of large data with the need for precision, selective inference is a good choice \cite{ref42}.  We only make definitive decisions on a selected subset, while the remaining subjects receive indecision. In other words, data that is hard to process will be held in abeyance, termed an "indecision" choice.

By introducing the indecision method, the False Selection Error Rate (\mbox{FSR}) can be effectively controlled. Indecision methods are not uncommon, for instance, in tumor image evaluations within the medical field. Various models can be employed to predict positive or negative outcomes, and images that the model cannot determine are left to experts for evaluation. This approach effectively reduces manual workload. \mbox{FSR} is a generic concept in selective inference, encompassing significant special cases such as the standard misclassification rate and false discovery rate. With this foundation, our objective is to maximize power as much as possible while maintaining control over the global \mbox{FSR} – that is, the expected fraction of erroneous decisions among all definitive choices.

In dealing with a standard multiple testing problem where the null distribution, denoted as $P_0$, is known, the Benjamini-Hochberg (BH) procedure \cite{ref27} ensures control over the FDR (False Discovery Rate). This method applies to finite samples and is uniform over all alternative distributions when the test statistics are either independent or satisfy the Positive Regression Dependency on a Subset (PRDS) property, according to \cite{ref18}, \cite{ref26}, and \cite{ref24}. 

There have been suggestions for variants of the BH procedure aimed at both easing the conservatism when the fraction of nulls does not approximate 1 – such as in the Storey-BH or Quantile-BH procedure \cite{ref28} and bolstering the FDR control under broader, dependent structures \ mention {ref30}.

The q-value method proposed by Storey \cite{ref12} is closely related to conformal inference \cite{ref1} and the BH procedure \cite{ref27}. Conformal inference is most commonly applied to outlier detection, as noted in \cite{ref6}, \cite{ref7}, and \cite{ref24}. However, it can also yield favorable outcomes for classification problems \cite{ref22}. These two methods can be effectively combined with contemporary machine learning and deep learning models, contributing to the interpretability of black-box models \cite{ref14}, \cite{ref21}, \cite{ref32}. The concept of local FDR \ inspires the q-value procedure we propose cite{ref4}, \cite{ref13}.

Unlike multiple testing, indecisions are not allowed in the conventional classification setup. However, decision errors, which can be very expensive to correct, are often unavoidable due to the intrinsic ambiguity of a classification model. The selective classification formulation provides a useful framework that trades off a fraction of indecision for fewer classification errors. The expected number of indecisions often reflects the difficulty of the task, as well as the degree of uncertainty in the decision-making process.

To maximize power, we focus on global \mbox{FSR}, and the problem of controlling \mbox{FSR} separately has been solved in the article \cite{ref9}. We will prove that this method is not optimal for controlling global \mbox{FSR}. In the paper \cite{ref10}, data is sampled from a compound of two normal distributions, and the author designed a clever shrinkage factor to control global FSR.

The second way to control the error rate for black-box models is to allow the classification result to contain multiple classes \cite{ref44}, \cite{ref2}. That is, assessing the probability that the true label of the sample in the predicted set is higher than a given threshold. Simultaneously, it is necessary to minimize the average number of predicted classes to maximize power.

In Section \ref{sectiontwo}, we will introduce the formulas used in this article. In Section \ref{sectionthree}, we provide proof of the optimal decision function under the oracle scenario. In Section \ref{sectionfour}, under the assumption of data exchangeability, we first propose a data-driven method. Next, when pi differs, we also provide an upper bound for the error rate under this method. Lastly, we consider the situation where different weights are assigned to the test data and provide a corresponding data-driven method along with its proof. In Section \ref{sectionfive}, we present a few simulations to demonstrate that our method can control error rate and compare to the previous FASI method, exhibits clear advantages. It indirectly validates the optimality results discussed in Section Three. Section \ref{sectionsix} concludes the article.

\section{Problem Formulation}\label{sectiontwo}
We divide the labeled data into two parts: train data and calibration data. The data to be predicted is called test data, and they do not have labels. The train data is used to train the model, and the calibration data is used to compare with the test data, which can correct the model to a certain extent.

In later sections, we will try to give an optimal score function in the context of global \mbox{FSR} control and correspondingly give a data-driven method. The following briefly introduces some symbolic representations: We assume there are $K$ classes, $\{ 1, 2, 3,..., K\}$. And $n$ samples $D = \{(X_{i},Y_{i}): 1\leq i \leq n\}$ is divided into a training set and a calibration set: $D = D^{train}\cup D^{cal}$, where $X_{i} \in R^{p}$ is a $p$-dimensional vector of features, and $Y_{i}$ is the real class, $Y_{i} \in \{ 1, 2, 3,..., K\}$. The (future) test set $\{X_{n+1},X_{n+2},...,X_{n+m}\}$ is recorded as $D^{test}$, and these data do not have labels. We conveniently record the sample size of calibration data and test data as $ n^{cal}$ and $n^{test}$. In the following content, to simplify the notation, we use $ i \in D^{cal}$ or $ i \in D^{test} $ to refer to $ (X_{i},Y_{i}) \in D^{cal}$ or $ X_{i} \in D^{test}$. 

We denote the predicted result of the $i$th sample as $\hat{Y}_{i}$. In the selective classification framework, if we have enough evidence that $X_{i} $ belongs to a class, then we predict $\hat{Y}_{i} \in \{1,2,...,K\}$, and if the class of the sample cannot be judged through the existing data, we correspondingly record $\hat{Y}_{i}$ as $0$. To focus on key ideas, we mainly consider the binary classification problem in our article. Therefore $ \hat{Y}_{i} \in \{0,1,2\} $, and $ Y_{i} \in \{1,2\}$. Since our model is designed to deal with black-box models, we use $S^1(X_{i})$ and $S^2(X_{i})$ to represent the scores for class 1 and class 2 for sample $X_i$ obtained from the model. Without loss of generality, we assume that the sum of  $S^1(X_{i})$ and $S^2(X_{i})$ is 1.

In the selective classification problem, the expression of \mbox{FSP} is $$ \mbox{FSP} = \frac{\sum_{i\in D^{test}} \mathbb{I}(\hat{Y}_{i}=Y_{i},\hat{Y}_{i}\neq 0)}{\sum_{i \in D^{test}} \mathbb{I}(\hat{Y}_{i}\neq 0)} , $$and \mbox{FSR} is the expectation of \mbox{FSP}.

Here, we no longer make too many assumptions about the distribution of $X$ but directly give an assumption on the distribution of $S^{i}(x)$. We define the distribution of $ S^{1}(x) $ and $S^{2}(x)$ in test data and calibration data as \begin{equation}\label{distribution_F}
	\begin{split}
		&S^{i}_{test}(x) \sim F^{i}_{test}(x) = \pi_{test}\cdot F^{i}_{1} + (1-\pi_{test})\cdot F^{i}_{2}, \\
		&S^{i}_{cal}(x) \sim F^{i}_{cal}(x) = \pi_{cal}\cdot F^{i}_{1} + (1-\pi_{cal})\cdot F^{i}_{2},
	\end{split}
\end{equation}where $ i = 1, 2$ and $\pi_{test}$/$\pi_{cal}$ is the probability of belonging to class 1 in test/calibration data, and $1-\pi_{test}$/$1-\pi_{cal}$ is the probability of belonging to class 2 in test/calibration data. $F^{j}_{k} $ is the conditional CDF of $S^{j}(x)$ given $Y = k$. Let $f^{j}_{k}$ be the probability density function of $ F^{j}_{k}$. We first assume that $\pi_{cal} = \pi_{test} $ between test and calibration data, then prove \mbox{FSR} can be controlled under the exchangeable condition. Next, assuming that $\pi_{cal} \neq \pi_{test}$ and proving the data-driven method also can control the \mbox{FSR}. In this paper, we have the following contributions:

\begin{itemize}
\item In the oracle procedure, we proved the optimality of $S(x) =  \max\{S_{1}(x), S_{2}(x)\}$ for the global \mbox{FSR} control. The optimal function of $K (K>2)$ class classification and the corresponding proof are in the appendix.
\item In the data-driven procedure, we prove our proposed algorithm controls the \mbox{FSR} under the exchangeability assumption among $ D^{cal} \cup D^{test} $. Define a new concept $p^{test}$ and give an estimation of it by Storey's estimator. We extend our data-driven method in binary classification to multi-classification problems in the appendix.
\item When the proportion of classes is different in calibration data and test data, We give an upper bound of our method. Through our numerical experiments, we can see that our method does control \mbox{FSR} and holds advantages in accuracy and power compared to other methods.
\item In the data-driven procedure, we provide another perspective on the martingale proof. We prove that when different weights are added to the test data and calibration data, the \mbox{FSR} also can be controlled.
\end{itemize}

\section{Oracle procedure and optimality of \mbox{FSR} control in binary classification problem}\label{sectionthree}
First, we consider the oracle situation and assume that the score obtained is the true probability of belonging to the class. We will prove the score function $S(x) = \max\{S^{1}(x),S^{2}(x)\}$ is optimal in the sense of \mbox{ETS} under the control of global \mbox{FSR}, where $S^{i} = P(Y = i| x)$. We make decisions by$$ \delta =  \underset{i}{{\arg\max} \, S^{i}(x)} \cdot \mathbb{I}(S(x) \geq t). $$ We denote the predicted value as $\hat{Y}$, and define \mbox{mFSR} and \mbox{ETS} for test data as:
$$\mbox{mFSR} = \frac{\mathbb{E}\{\sum_{i\in D^{test}}(\hat{Y}_{i} \neq Y_{i},\hat{Y}_{i} \neq 0)\}}{\mathbb{E}\{\sum_{i\in D^{test}}(\hat{Y}_{i} \neq 0)\}},$$ 
$$ \mbox{ETS} = \sum_{i\in D^{test}}(\hat{Y}_{i} = Y_{i},\hat{Y}_{i} \neq 0).$$
We cannot express the decision rule $ \delta $ only by $S$ because $S$ cannot represent all information of samples. We still use $S^{i}$ here for convenience.

\begin{theorem}\label{optimaloracle}
	Under the assumption of the distribution of $S^{1}, S^{2}$ of test data and calibration data, and $\alpha$ has been chosen in advance. Let $ D_{\alpha}$ denote the selection rules that satisfy $\mbox{mFSR} \leq \alpha$. Let $\mbox{ETS}_{\delta}$ denote the $\mbox{ETS}$ of an arbitrary decision rule $\delta$. Then, the oracle procedure is optimal in the sense that $\mbox{ETS}_{\delta_{OR}}\leq \mbox{ETS}_{\delta} $ for any $ \delta \in D_{\alpha}$ and finite sample size $n^{test}$.
\end{theorem}

From the definition of \mbox{mFSR}, there is$$ \mbox{mFSR}(t) = P(Y\neq\hat{Y}|S(x)\geq t ).$$ Traditionally, we still use $Q(t)$ to denote $\mbox{mFSR}(t)$. We first prove that the $Q(t)$ is monotonically decreasing in the oracle situation. 
\begin{equation} \label{eq1}
	\begin{split}
		Q(t)&=P(Y\neq \hat{Y}|S(x)\geq t)    \\
		&=\frac{E(\sum_{j\in D^{test}}((1-S_{j})\mathbb{I}(S_{j} \geq t)))}{E(\sum_{j\in D^{test}}\mathbb{I}(S_{j}\geq t))},
	\end{split}
\end{equation}where $ S_{j} = S(X_{j}) = \max\{ S^{1}(X_{j}),  S^{2}(X_{j}) \}$.

\begin{lemma} \label{lemma1}
	Suppose $ p(x)$ is a non-negative and bounded function, and $f(x)$ is a monotonically decreasing function. Then function \eqref{eq2} decreases monotonically, and the function $Q(t)$ defined by \eqref{eq1} decreases monotonically. 
	\begin{equation}\label{eq2}
		g(t) = \frac{\int_{t}^{1}f(x)p(x)dx}{\int_{t}^{1}p(x)dx}
	\end{equation}
\end{lemma}

Due to the monotonicity of $Q(t)$, we have the corresponding $ Q^{-1}(\alpha)$ for any threshold $ \alpha$ set in advance. So the oracle rule can be written as
\begin{equation}\label{delta}
	\delta^{n+j}(\alpha) =  \underset{i}{{\arg\max} \, S^{i}_{n+j}} \cdot \mathbb{I}(S_{n+j} \geq Q^{-1}(\alpha)).
\end{equation} Where no ambiguity arises, we denote $ {\arg\max}_{i} \, S_{n+j}^{i}$ as $ Y^{\prime}_{n+j}$ and $ \delta^{n+j}(\alpha)$ as $ \hat{Y}_{n+j}$.

And it is easy to see that $$\delta(\alpha)  = 1\cdot \mathbb{I}(S^{1}\geq Q^{-1}(\alpha)) + 2\cdot \mathbb{I}(S^{2}\geq Q^{-1}(\alpha))$$ are equivalent to the expression \eqref{delta}. Since $S^{i}(x)$ satisfies the equation $ S^{1}(x) + S^{2}(x) = 1$, $Q(\alpha)$ must also be a number greater than 0.5 for $S(x) \geq 0.5$, so we don’t have to consider overlapping situations.

\begin{remark}
    Let us use another angle to understand the problem. $S^{i}(x)$ itself expresses the possibility of $x$ belonging to class $i$. If we only classify according to the large or small value of $S_{i}$, that is, ${\arg\max}_{i} \, S^{i} $, it is logically impossible to control the error rate because of the model-free assumption. So we add the screening of the value of $ S $, based on $ {\arg\max}_{i} \, S^{i}$. When the value of $S$ is large enough, we classify according to ${\arg\max_{i}} \, S^{i}$. That is the explanation of $ \delta =  {\arg\max}_{i} \, S^{i} \cdot \mathbb{I}(S \geq t) $. 
\end{remark}

\section{The data-driven procedure of \mbox{FSR} control in binary classification problem}\label{sectionfour}
\subsection{Data-driven procedure}
Just like the data-driven procedure in FASI, we use the martingale method to control the \mbox{FSR}. The difference with the oracle situation is that the score function here is no longer $ S^{i}(x) = P(y = i| x)$, but a heuristic score, and we denote it as $\hat{S}^{i}$. So the corresponding $Q(t)$ is no longer monotonic. Here we learn from the idea of the q-value in \cite{ref11}.

We again explain the notation that we will use: we denote the calibration set as $ D^{cal}$, the test set as $D^{test}$, where $ n^{cal}$ represents the number of samples in the calibration set and $ n^{test}$ represents the number of samples in the test set.

\begin{assumption}\label{assumption1}
	The groups $\{(X_{i}, Y_{i}): X_i \in D^{cal} \cup D^{test}\} $are exchangeable.
\end{assumption}

Our decision rule is $$ \delta^{j}(t) =  \underset{i}{{\arg\max} \, \hat{S}^{i}_{j}} \cdot \mathbb{I}(\hat{S}_{j} \geq t),$$where $\hat{S_{j}} = \max\{\hat{S}^{1}(X_{j}),\hat{S}^{2}(X_{j})\}$, $\hat{S}^{i}_{j}$ is the score of class $i$ for sample $j$ from the black-box model. The estimated false discovery proportion (\mbox{FSP}), as a function of $t$, is given by
\begin{equation}\label{decision}
	\hat{Q}(t) = \frac{\frac{1}{n^{cal}+1}\{\sum_{i\in D^{cal}}\mathbb{I}(\hat{S_{i}}\geq t,\hat{Y}_{i} \neq Y_{i})+1\}}{\frac{1}{n^{test}}\{\sum_{n+j \in D^{test}}\mathbb{I}(\hat{S}_{n+j}\geq t)\}\bigvee 1},
\end{equation}
where $\hat{Y}_{i}$ denote the predicted class of sample $i$, $Y_{i}$ denote the real class of sample $i$. We choose the smallest $t$ such that the estimated \mbox{FSP} is less than $\alpha$. Define
\begin{equation}\label{tau}
	 \tau = \hat{Q}^{-1}(\alpha) = \inf \{t: \hat{Q}(t)\leq \alpha\},
\end{equation}
and \begin{equation}\label{R_value}
    \hat{R}_{n+j} = \inf_{t\leq \hat{s}}\{\hat{Q}(t)\},
\end{equation} where $ \hat{s} = \hat{S}(X_{n+j})$. Therefore $$ \mathbb{I}(\hat{R}_{n+j} \leq \alpha) \Leftrightarrow \mathbb{I}(\hat{S}_{n+j}\leq \tau).$$ So decision rule can be written as: for a given $\alpha$, $$\delta^{n+j} =  \underset{i}{{\arg\max} \, \hat{S}^{i}_{n+j}} \cdot \mathbb{I}(\hat{R}_{n+j} \leq \alpha).$$ 

\renewcommand{\algorithmicrequire}{\textbf{Input:}}
\renewcommand{\algorithmicensure}{\textbf{Output:}}

\begin{algorithm}\label{algori}
\caption{\mbox{FSR} control procedure on binary classification.}
\label{alg:1}
\begin{algorithmic}[1]
\REQUIRE Existing data $D$ and its real class, test data $D^{test}$.
\ENSURE Classification result of test data.
\STATE Randomly split $D$ into $D^{train}$ and $D^{cal}$. 
\STATE Train a black-box model only on $D^{train}$ to get score function $S^{1}(x)$, $S^{2}(x)$.
\STATE Predict base scores for all observations in $D^{test}$ and $D^{cal}$.
\STATE Compute the q-value for all test data using equation \eqref{decision}.
\STATE Compute the R-value for all test data using equation \eqref{R_value}.
\STATE Threshold the R-value at a user-specified level $\alpha$, assigning an observation in $ D^{test}$ to class $\hat{Y}_{i}$ if $\hat{R}_{i}\leq \alpha$, where$$\hat{Y}_{i} =  \underset{j}{{\arg\max} \, S^{j}_{i}} \cdot\mathbb{I}(\hat{R}_{i}\leq\alpha).$$
\STATE Return an indecision result on all remaining observations where $\hat{R}_{i} \geq \alpha$.
\end{algorithmic}
\end{algorithm}

\begin{theorem}\label{data_driven}
	We denote $W^{test}$/$W^{cal}$ as the samples whose real class corresponds to a lower score in the $\{S^{1}, S^{2}\}$ in the test/calibration data. Define $$p_{test} = \frac{|W^{test}|}{n^{test}} ,\ p_{cal} = \frac{|W^{cal}|}{n^{cal}} ,\ \gamma = \mathbb{E}(\frac{p_{test}}{p_{cal}}).$$Assume there exist $x_{0} \neq x_{1}$, s.t. $F_{1}^{1}(x_{0})\neq F_{1}^{2}(x_{0}) $ and $ F_{1}^{1}(x_{1})\neq F_{1}^{2}(x_{1}) $. Then, under Assumption \ref{assumption1}, the Algorithm \ref{alg:1} with R-value formula \eqref{decision} can control \mbox{FSR} at $ \alpha $ precisely. 
\end{theorem}

\subsection{Control \mbox{FSR} under different sparsity}
To maximize the method power, we need to estimate $p_{test}$ when data are not exchangeable. From the assumption on distribution \eqref{distribution_F}, we know:\begin{equation}
	\begin{split}
		&S^{1}_{test}(x) \sim F^{1}_{test}(x) = \pi_{test}\cdot F^{1}_{1} + (1-\pi_{test})\cdot F^{1}_{2}, \\ 
        & S^{2}_{test}(x) \sim F^{2}_{test}(x) = \pi_{test}\cdot F^{2}_{1} + (1-\pi_{test})\cdot F^{2}_{2}.
	\end{split}
\end{equation}
We are considering a binary classification problem. When treating $n^{test}$ as a fixed constant, the equation can be simplified to:
\begin{equation}\label{p_test}
	\begin{split}
		\mathbb{E}(p_{test}) &= \mathbb{E}( |W^{test}|/n^{test}) \\
		&= \frac{\mathbb{E} [\sum_{i\in D^{test}} \{ \mathbb{I}(S^{1}(x_{i})\geq 0.5, c = 2) + \mathbb{I}(S^{1}(x_{i})< 0.5, c = 1) = 1\}]}{n^{test}}.
	\end{split}
\end{equation}
According to Chebyshev's law of large numbers, the above formula converges to the following formula with probability 1. \begin{equation}
	\begin{split}
		\mathbb{E}(p_{test}) &= \mathbb{P}(S^{1}(x_{i})\geq 0.5, c = 2) + \mathbb{P}(S^{1}(x_{i})< 0.5, c = 1) \\
		&= (1-\pi_{test})\cdot(1 - F^{1}_{2}(0.5)) + \pi_{test} \cdot F^{1}_{1}(0.5).
	\end{split}
\end{equation}
So we only need to estimate $\pi_{test} $ and $ F_{1}^{1}$. And $ F_{1}^{1}$ can be estimated from calibration data, $\pi_{test}$ can be estimated by Storey's estimator \cite{ref11} used on test data, that is: we define $$W^{test}(\lambda) = \# \{i \in D^{test}: S^{1}_{i} > \lambda\},$$ and from the empirical CDF the expected number of $S^{1}_{i}$ should be $n^{test}\cdot \pi_{test} \cdot (1- F_{1}^{1}(\lambda))$. Setting expected equal to observed, we obtain: \begin{equation}\label{pi_test_storey}
    \hat{\pi}_{test}(\lambda) = \frac{W^{test}(\lambda)}{n^{test}(1-\hat{F}_{1}^{1}(\lambda))},
\end{equation} and $\mathbb{E} \{\hat{\pi}(\lambda)\}\geq \pi $. We plug the estimation of $\pi_{test}$ into equation \eqref{p_test}, so as to get the estimation of $\mathbb{E}(p_{test})$:
\begin{equation}\label{p_test_estima}
    \begin{split}
        \hat{p}_{test} &= (1-\hat{\pi}_{test}(\lambda))\cdot(1 - \hat{F}^{1}_{2}(0.5)) + \hat{\pi}_{test}(\lambda) \cdot \hat{F}^{1}_{1}(0.5) \\
        &= 1 -  \hat{F}^{1}_{2}(0.5) + \hat{\pi}_{test}(\lambda)\cdot( \hat{F}^{1}_{2}(0.5) - 1 + \hat{F}^{1}_{1}(0.5)),
    \end{split}
\end{equation}where the $\hat{F}^{i}_{j}(0.5)$ is calculated by calibration data, and  we take $\lambda$ as $\tau$, calculated by \eqref{tau}.

\begin{assumption}\label{assumption3}
    $\pi_{test} \neq \pi_{cal}$ but the conditional distribution of the same class is consistent in test data and calibration data.
\end{assumption}

Here, we consider the case where the proportions of class 1 and class 2 in test and calibration data are different. So the $\pi_{test}$ is estimated using Storey's estimator on test data, and the $\pi_{cal}$ can be estimated directly using the empirical distribution of calibration data.

\begin{lemma}\label{lowerbound}
    In the procedure of estimating $\pi_{test}$, we use Storey's estimator for $ S^{1}$ when $$ F^{1}_{2}(0.5) + F^{1}_{1}(0.5) -1 \geq 0.$$ That is:  $$W^{test}(\lambda) = \# \{i \in D^{test}: S^{1}_{i} > \lambda\}, \ \hat{\pi}_{test}(\lambda) = \frac{W^{test}(\lambda)}{n^{test}(1-\hat{F}_{1}^{1}(\lambda))}.$$ Otherwise, we use Storey's estimator for $ S^{2}$. Under the steps above, we  have $\mathbb{E}(p_{test}) \leq \mathbb{E}\{\hat{p}_{test}\}$ for any $\lambda$.
\end{lemma}
According to the Lemma \ref{lowerbound}, without loss of generality, we can assume that $F^{1}_{2}(0.5) + F^{1}_{1}(0.5) -1 \geq 0 $, and Storey's estimator is used for $S^1$. Therefore we define:
\begin{equation}\label{tilde_Q}
    \begin{split}
        \tilde Q(t) &= \frac{\hat{p}_{test}}{\hat{p}_{cal}} \cdot \hat{Q}(t) \\
        &= \frac{\hat{p}_{test}}{\hat{p}_{cal}} \cdot \frac{\frac{1}{n^{cal}+1} \{\sum_{i\in D^{cal}}\mathbb{I}(\hat{S_{i}}\geq t,\hat{Y}_{i} \neq Y_{i})+1\}}{\frac{1}{n^{test}}\{\sum_{n+j \in D^{test}}\mathbb{I}(\hat{S}_{n+j}\geq t)\}\bigvee 1}.
    \end{split}
\end{equation} From the proof above, We choose the smallest $t$ such that the estimated \mbox{FSP} is less than $\alpha$. Define
\begin{equation}\label{tau_}
	 \tilde \tau =  \tilde Q^{-1}(\alpha) = \inf\{t: \tilde Q(t)\leq \alpha\},
\end{equation}
and \begin{equation}\label{R_value_}
    \tilde R_{n+j} = \inf_{t\leq \hat{s}}\{\tilde Q(t)\},
\end{equation} where $ \hat{s} = \hat{S}(X_{n+j} = x)$. So, the new decision rule can be written as: for a given $\alpha$, 
$$\tilde \delta^{n+j} =  \underset{i}{{\arg\max} \, S^{i}_{n+j}} \cdot \mathbb{I}(\tilde R_{n+j} \leq \alpha).$$

\renewcommand{\algorithmicrequire}{\textbf{Input:}}
\renewcommand{\algorithmicensure}{\textbf{Output:}}

\begin{algorithm}
\caption{\mbox{FSR} control procedure with different $\pi$.}
\label{alg:2}
\begin{algorithmic}[1]
\REQUIRE Existing data $D$ and its real class, test data $D^{test}$.
\ENSURE Classification result of test data $\tilde Y_{i}$.
\STATE Randomly split $D$ into $D^{train}$ and $D^{cal}$. 
\STATE Train a black-box model only on $D^{train}$ to get score function $S^{1}(x)$, $S^{2}(x)$.
\STATE Predict base scores for all observations in $D^{test}$ and $D^{cal}$.
\STATE Compute the q-value, using equation \eqref{decision} for $D^{test}$, and calculate $\tau$ by \eqref{tau}.
\STATE Calculate $\hat{\pi}_{test}$ by formula \eqref{pi_test_storey} and take $\lambda$ as $\tau$, calculate $\hat{p}_{cal}$ by its empirical distribution$$\hat{p}_{cal} =  \frac{\{\sum_{i\in D^{test}} [ \mathbb{I}(S^{1}(x_{i})\geq 0.5, c = 2) + \mathbb{I}(S^{1}(x_{i})< 0.5, c = 1) = 1]\}}{n^{test}}.$$
\STATE Estimate  $\hat{F}^{1}_{2}(0.5)$, $\hat{F}^{1}_{2}(0.5)$ and $ \hat{F}^{1}_{1}(0.5))$ by calibration data and plug in formula \eqref{p_test_estima} to get $\hat{p}_{test}$. 
\STATE Compute the new q-value $\tilde Q_{i}$ for all test data by \eqref{tilde_Q} and $\tilde\tau$ by \eqref{tau_}.
\STATE Compute the new R-value $ \tilde R_{i}$ by \eqref{R_value_} for all test data based on $\tilde Q_{i}$.
\STATE Threshold the new R-value at a user-specified level $\alpha$, assigning an observation in $ D^{test}$ to class $\tilde Y_{i}$ if $\tilde R_{i}\leq \alpha$, where $$\tilde Y_{i} =  \underset{j}{{\arg\max} \, S^{j}_{i}} \cdot\mathbb{I}(\tilde R_{i}\leq\alpha).  $$
\STATE Return indecision on all remaining observations where $\tilde R_{i} \geq \alpha$.
\end{algorithmic}
\end{algorithm}
\mbox{FSP} of the proposed algorithm is given by \begin{equation}\label{marting}
    \mbox{FSP}(\tilde \tau) = \frac{V^{test}(\tilde \tau)}{R^{test}(\tilde \tau)\bigvee 1}.
\end{equation} And we have the following theorem.

\begin{theorem}\label{datadrivenpidifferent}
Denote $$ C = \frac{\mathbb{E}(p_{cal})}{\mathbb{E}(p_{test})} \cdot( \frac{\pi_{test}}{\pi_{cal}} + \frac{1 - \pi_{test}}{1-\pi_{cal}}). $$ Then under Assumption \ref{assumption3} that $\pi_{test}$ is different from $\pi_{cal}$ and the conditional distribution of the same class is consistent in test data and calibration data. The Algorithm \ref{alg:2} with R-value \eqref{R_value_} can control \mbox{FSR} at $C\alpha$, where$$ \mbox{FSR} = \mathbb{E} \left\{\frac{\sum_{i\in D^{test}}(\tilde Y_{i} \neq Y_{i},\tilde Y_{i} \neq 0)}{\sum_{i\in D^{test}}(\tilde Y_{i} \neq 0)\bigvee 1}\right\}.$$
\end{theorem}
We denote \begin{equation}
    \begin{split}
         V^{test}(\tilde \tau) &= \# \{i: S(x_{i}) \geq \tilde \tau, Y_i \neq \hat{Y}_{i}, x_{i} \in D^{test} \}, \\
         V^{cal}(\tilde \tau) &= \# \{i: S(x_{i}) \geq \tilde \tau, Y_i \neq \hat{Y}_{i}, x_{i} \in D^{cal} \}, \\
         R^{test}(\tilde \tau) &= \# \{i: S(x_{i}) \geq \tilde \tau,  x_{i} \in D^{test} \}, \\
         R^{cal}(\tilde \tau) &= \# \{i: S(x_{i}) \geq \tilde \tau,  x_{i} \in D^{cal} \}. \\
    \end{split}
\end{equation}To illustrate the theorem, we expand the expression of $\mbox{FSP}(\tilde \tau)$ above
\begin{equation}
    \begin{split}
		\mbox{FSP}(\tilde \tau) &= \frac{V^{test}(\tilde \tau)}{V^{cal}(\tilde \tau)+1}\cdot \frac{V^{cal}(\tilde \tau)+1}{R^{test}(\tilde \tau)\bigvee 1} \\
		=& \tilde Q(\tilde \tau)\cdot \frac{n^{cal}+1}{n^{test}} \cdot \frac{\hat{p}_{cal}}{\hat{p}_{test}} \cdot\frac{V^{test}(\tilde \tau)}{V^{cal}(\tilde \tau)+1} \\
        \leq & \alpha \cdot \frac{n^{cal} + 1}{ n^{test}} \cdot \frac{c_0}{c_1 + c_2 \cdot \hat{\pi}_{test}} \cdot \frac{V^{test}(\tilde \tau)}{V^{cal}(\tilde \tau) +1},
    \end{split}
\end{equation}
where $ c_{0} = | W^{cal}|/ | D^{cal}| ,\ c_{1} = 1 - \hat{F}^{1}_{2}(0.5) ,\  c_{2} = \hat{F}_{2}^{1}(0.5) - 1 + \hat{F}^{1}_{1}(0.5)$. We prove the following lemma:
\begin{lemma}\label{exchangeable}
    Let $(\Omega, \mathcal{F}, \mathbb{P})$ be a triple, and $X$ is a random variable with $\mathrm{E}(|X|)<\infty$. Let $\mathcal{G}$ and $\mathcal{H}$ be a sub-$\sigma$-algebra of $\mathcal{F}$.
    
    i. If $\mathcal{H}$ is independent of $\sigma(\sigma(X), \mathcal{G})$ then $
    \mathbb{E}[X|\sigma(\mathcal{G}, \mathcal{H})]=\mathbb{E}(X|\mathcal{G}),  \text { a.s. }$

    ii.Let $ \{s_{1},s_2,...,s_{n}\}$ be a series random variables which satisfy $a \leq s_{1}\leq s_2\leq ...\leq s_n\leq b$ and exchangeability. X is defined by a subsequence of $\{s_{k}\}_{1\leq k\leq n}$, and independent of 
 b. If $\mathcal{H} = \sigma(s_{i})_{k+1 \leq i\leq n}$, $\mathcal{G} = \sigma(s_{i})_{k\leq i\leq n}$, where $s_{k}$ is independent of $X$, then $$\mathbb{E}[X|\mathcal{G}]=\mathbb{E}(X|\mathcal{H}), \quad \text { a.s. }$$
    \end{lemma}
    
    \begin{lemma}\label{chengji}
    if $(X_{n},\sigma(s_{j})_{j\leq n}, n\geq 0)$ and  $(Y_{n},\sigma(s_{j})_{j\leq n}, n\geq 0)$ are martingales. $ |Y_{n}|$ and $| X_{n} |$ has a finite upper bound $M$, then $(X_{n}\cdot Y_{n}, \sigma(s_{j})_{j\leq n}, n\geq 0)$ is also a martingale.
\end{lemma}

When the exchangeability is not satisfied, the expression \ref{marting} is not a martingale anymore so we can divide it by \begin{equation}
    \frac{V^{test}(\tau)}{V^{cal}(\tau) + 1} \leq \frac{V_{1}^{test}(\tau)}{V_{1}^{cal}(\tau) + 1} + \frac{V_{2}^{test}(\tau)}{V_{2}^{cal}(\tau) + 1},
\end{equation}
where $V_{1}^{test}(\tau)$ are the samples that belong to class 2 but are classified to class 1 when the threshold is $\tau$, and $V_{2}^{test}(\tau)$ are the samples that belong to class 1 but are classified to class 2 when the threshold is $\tau$.

\begin{lemma}\label{piestimate}
We force the following discrete-time filtration that describes the misclassification process:$$ F_{k} = \{ \Delta(V^{test}(s_{k}),V^{cal}(s_{k}))\}_{t_{l}\leq s_{k}\leq t},$$ where $s_{k}$ corresponds to the threshold(time) when exactly k subjects, combining the subjects in both $ D^{cal}$ and $ D^{test}$, are misclassified, $n = n^{cal} + n^{test}$. Then $$\frac{V_{1}^{test}(\tau)}{V_{1}^{cal}(\tau) + 1} , \ \frac{V_{2}^{test}(\tau)}{V_{2}^{cal}(\tau) + 1}$$ are both martingales.
\end{lemma}

\begin{remark}
    From Theorem \ref{datadrivenpidifferent}, we give an upper bound for \mbox{FSR} control when data are not exchangeable between $D^{cal}$ and $D^{test}$, but the upper bound is influenced by different $\pi$. When $\pi$ is extremely large or small, this can lead to a large upper bound. Therefore, our future work will propose a new q-value method to estimate \mbox{mFSR} to give a more precise control.
\end{remark}

\subsection{Weighted data-driven procedure}
We still consider the situation of Assumption \ref{assumption1} here. When the reliability of calibration data is lower than that of test data, we can improve our method by assigning more weight to the test data. In this case, we first prove a new mirror process to demonstrate that it is a martingale. Then, we propose a new algorithm and provide a theorem to ensure its reliability.
\begin{corollary}\label{weightmartingale}
	we force the following discrete-time filtration that describes the misclassification process:$$ F_{k} = \{ \Delta(R^{test}(s_{k}),R^{cal}(s_{k}))\}_{t_{l}\leq s_{k}\leq t},$$ where $s_{k}$ corresponds to the threshold(time) when exactly k subjects, combining the subjects in both $ D^{cal}$ and $ D^{test}$ are rejected. Then $$ \frac{R^{cal}(s_{k})+K\cdot R^{test}(s_{k}) + K}{R^{test}(s_{k} )+ 1 }$$ is a martingale for any positive integer $K$.
\end{corollary}

 Therefore we define:
\begin{equation}\label{Q_W}
         \hat{Q}_{w}(t, K) = \frac{\frac{1}{n^{cal}+1}\{\sum_{i\in D^{cal}}\mathbb{I}(S_{i}\geq t, \hat{Y}_{i}\neq Y_{i}) + 1\}}{\frac{1}{n^{cal} + K\cdot(n^{test}+1)}\{\sum_{i\in D^{cal}}\mathbb{I}(S_{i}\geq t) + K\cdot (\sum_{i\in D^{test}}\mathbb{I}(S_{i}\geq t ) + 1)\}}. 
\end{equation} We take $K$ as a constant as the ratio between the weights of test data and calibration data. Choose the smallest $t$ such that the estimated \mbox{FSP} is less than $\alpha$. Define
\begin{equation}\label{tau_W}
	 \hat{\tau}_{w} =  \hat{Q}^{-1}_{w}(\alpha) = \inf \{t: \hat{Q}_{w}(t)\leq \alpha\},
\end{equation}
and \begin{equation}\label{R_value_W}
    R^{w}_{n+j} = \inf_{t\leq \hat{s}}\{\hat{Q}_{w}(t)\},
\end{equation} where $ \hat{s} = \hat{S}(X_{n+j} = x)$. So, the decision rule can be written as: for a given $\alpha$, \begin{equation}\label{decision_W}
    \delta_{w}^{n+j} =  \underset{i}{{\arg\max} \, S^{i}_{n+j}} \cdot \mathbb{I}(R^{w}_{n+j} \leq \alpha).
\end{equation}

\renewcommand{\algorithmicrequire}{\textbf{Input:}}
\renewcommand{\algorithmicensure}{\textbf{Output:}}

\begin{algorithm}\label{algorith}
\caption{\mbox{FSR} control procedure with different weights on test data.}
\label{alg:5}
\begin{algorithmic}[1]
\REQUIRE Existing data $D$ and its real class, test data $D^{test}$.
\ENSURE Classification result of test data.
\STATE Randomly split $D$ into $D^{train}$ and $D^{cal}$.
\STATE Train a black-box model only on $D^{train}$ to get the score function $S^{1}(x)$, $S^{2}(x)$.
\STATE Predict base scores for all observations in $D^{test}$ and $D^{cal}$.
\STATE Take $K$ as a constant as the ratio of the weights of test data and calibration data.
\STATE Compute the q-value for all test data using equation \eqref{Q_W}.
\STATE Compute the R-value for all test data using equation \eqref{R_value_W}.
\STATE Threshold the R-value at a user-specified level $\alpha$, assigning an observation in $ D^{test}$ to class $\hat{Y}_{i}$ if $\hat{R}_{i}\leq \alpha$, where $$\hat{Y}_{i} =  \underset{j}{{\arg\max} \, S^{j}_{i}} \cdot\mathbb{I}(R^{w}_{i}\leq\alpha).$$
\STATE Return an indecision result on all remaining observations where $\hat{R}_{i} \geq \alpha$.
\end{algorithmic}
\end{algorithm}

\begin{remark}
	In the proof of martingale, the structure is the most important thing. Any expression structure similar to \eqref{mart} can be proved as a martingale. A martingale plus a constant is still a martingale. Therefore, in the article FASI, $$ \frac{R^{cal}(s_{k}) + R^{test}(s_{k}) + 1 }{ R^{test}(s_{k}) + 1}$$ can be written as $$ 1 + \frac{R^{cal}(s_{k})}{ R^{test}(s_{k}) + 1}.$$ The structure of adding one to the denominator is the same as the expression \ref{marting}. 
\end{remark}

\begin{theorem}\label{datadrivenweighted}
	Under the Assumption \ref{assumption1}, assigning different weights to calibration data and test data according to different scenarios, we can get different weighted selective classification methods, and they can also control the global \mbox{FSR}, for all positive integers $K$ by Algorithm \ref{alg:5}.
\end{theorem}

Therefore, taking different $K$ always controls the global \mbox{FSR}, and the degree of control is the same. Observing the role of $ K$ in the $\hat{Q}_{w}(t)$ formula shows that the larger the value of $K$, the more times the test data is reused, that is, the higher the weight of the test data. Therefore, this method is suitable for when the reliability of the calibration data is not high, including when the calibration data is simulated data, or there exists noisy data, etc. Then, we can give the test data a higher weight by adjusting the value of $K$.

\section{Numerical Simulation}\label{sectionfive}
This section presents the results from two simulation scenarios comparing FASI to our algorithm. We illustrate that our algorithm can control \mbox{FSR} more accurately so that the algorithm has greater power. We denote the distribution of $S^{1}$ in test data as $$F(\cdot) = \pi_{test} \cdot F^{1}_{1}(x) + (1-\pi_{test})\cdot F^{1}_{2}(x), $$ and the distribution of calibration data as $$F(\cdot) = \pi_{cal} \cdot F^{1}_{1}(x) + (1-\pi_{cal})\cdot F^{1}_{2}(x). $$ We simulate 100 data sets and apply our algorithm and FASI with R-values defined in \eqref{R_value_} at \mbox{FSR} level 0.1 to the simulated data sets. We take the distribution of $F^{1}_{1}$ and $F^{1}_{2}$ as follows:$$F^{1}_{1}=\mathcal{N}\left(\frac{3}{8}, \frac{1}{64}\right), \quad F^{1}_{2} = \mathcal{N}\left(\frac{5}{8}, \frac{1}{64}\right).$$ And taken $n^{cal} = 1500$ and $n^{test} = 1000$. When the sample score is less than 0 or greater than 1, we take it as 0 or 1.

\begin{figure}
\includegraphics[scale=0.37]{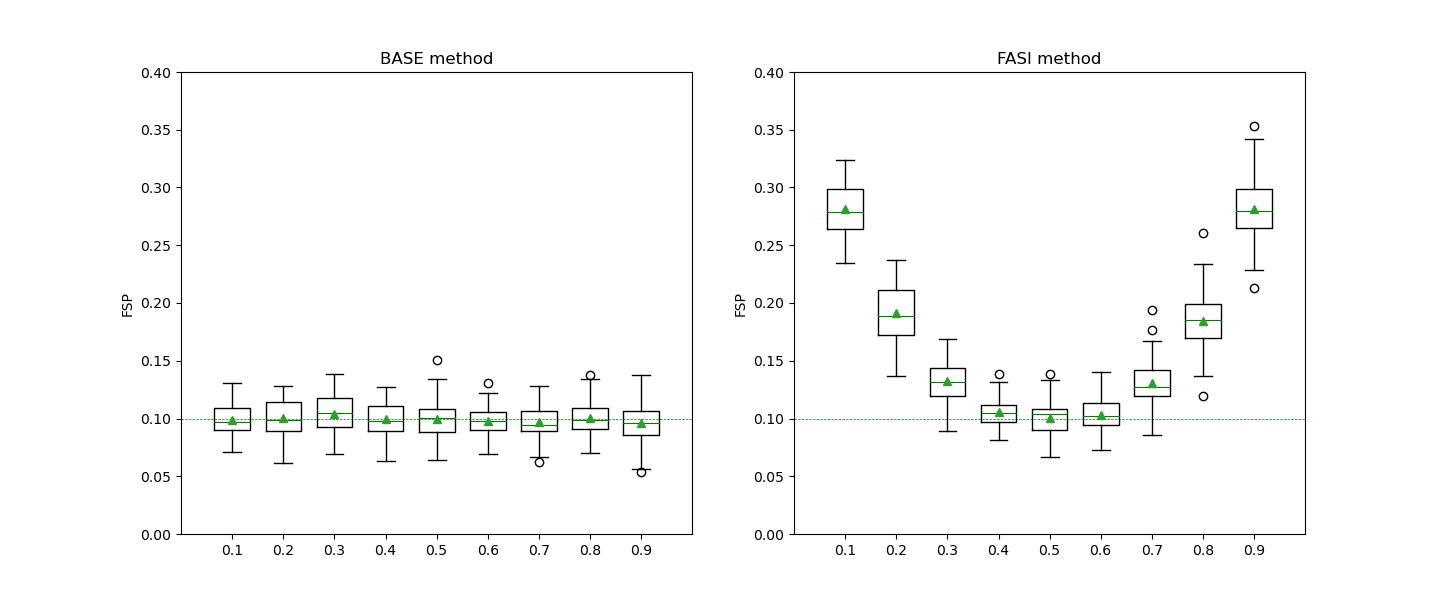}
\caption{The horizontal axis is $\pi_{cal}$, and $\pi_{test} $ is taken as 0.5. The picture on the left shows the change of \mbox{FSP} with $\pi_{cal}$ in our base method. The picture on the right shows the change of \mbox{FSP} with $\pi_{cal}$ in the FASI method.}
\label{pic1}
\end{figure}

Next, we verified through multiple experiments that our base method can control \mbox{FSR} no matter what value $\pi$ takes and compared it with FASI. Although FASI can control the error rate of each class, the overall error rate will be inflated or conservative. It is difficult for FASI to accurately control global \mbox{FSR} near the predetermined threshold. As shown in fig \ref{pic1}, we fixed $\pi_{test} $ at 0.5 and took $\pi_{cal} $ as $ \{ 0.1, 0.2,..,0.8, 0.9 \} $ respectively. Through experiments, we found that when the value of $\pi_{cal}$ gradually moves away from 0.5, the \mbox{FSP} of the FASI algorithm gradually becomes uncontrollable, but the \mbox{FSP} of our algorithm is still stable.

\begin{figure}
\includegraphics[scale=0.37]{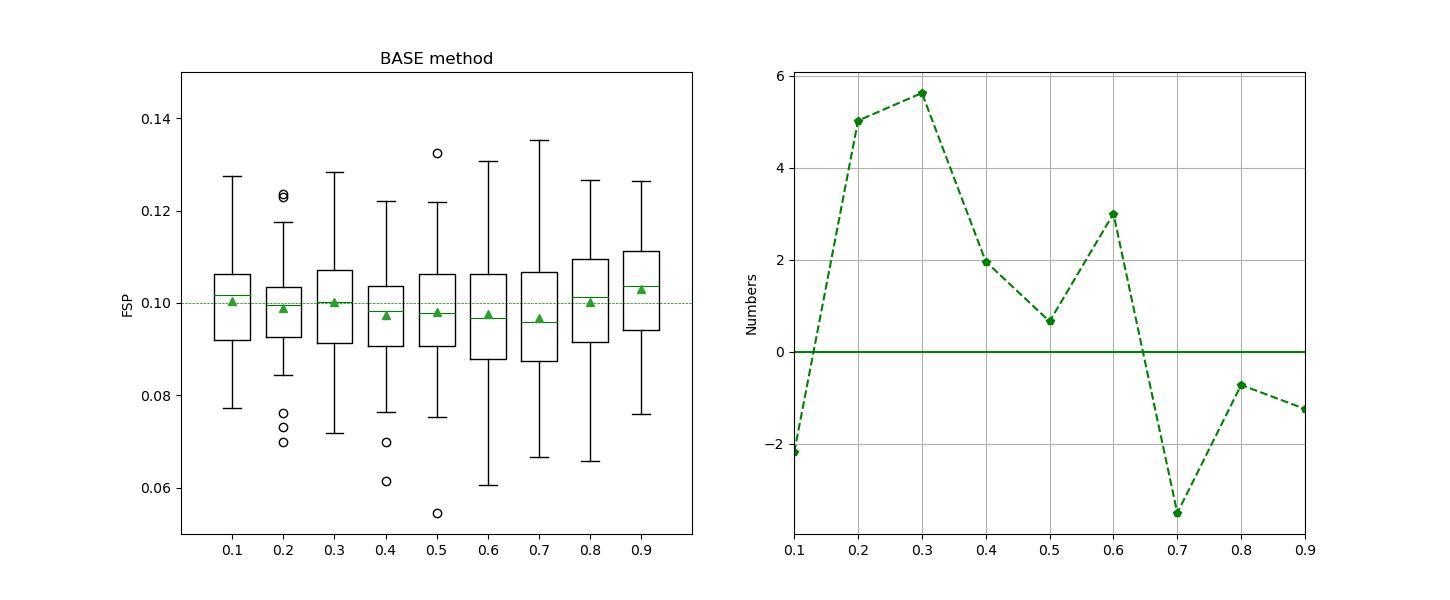}
\caption{The horizontal axis is $\pi_{test}$, and $\pi_{cal} $ is taken as 0.5. The picture on the left shows the change of \mbox{FSP} with $\pi_{test}$ in our base method, and the image on the right shows how the number of rejections of the base method minus the number of rejections of the FASI method changes with $\pi_{test}$ when taking $n^{test}$ as 1000. }
\label{pic2}
\end{figure}
In the second group experiment, we set $\pi_{cal}$ fixed at 0.5, and $\pi_{test}$ is set to $\{0.1, 0.2,...,0.8, 0.9\}$ respectively, we will find that the \mbox{FSP} of the base method and the FASI method are still stable near the threshold, and comparing the power of the two algorithms, we will find that the overall result of the base method is slightly better than the FASI method, which also verifies the optimality of the score function we proved earlier, as shown in fig \ref{pic2}.

\begin{figure}
\includegraphics[scale=0.4]{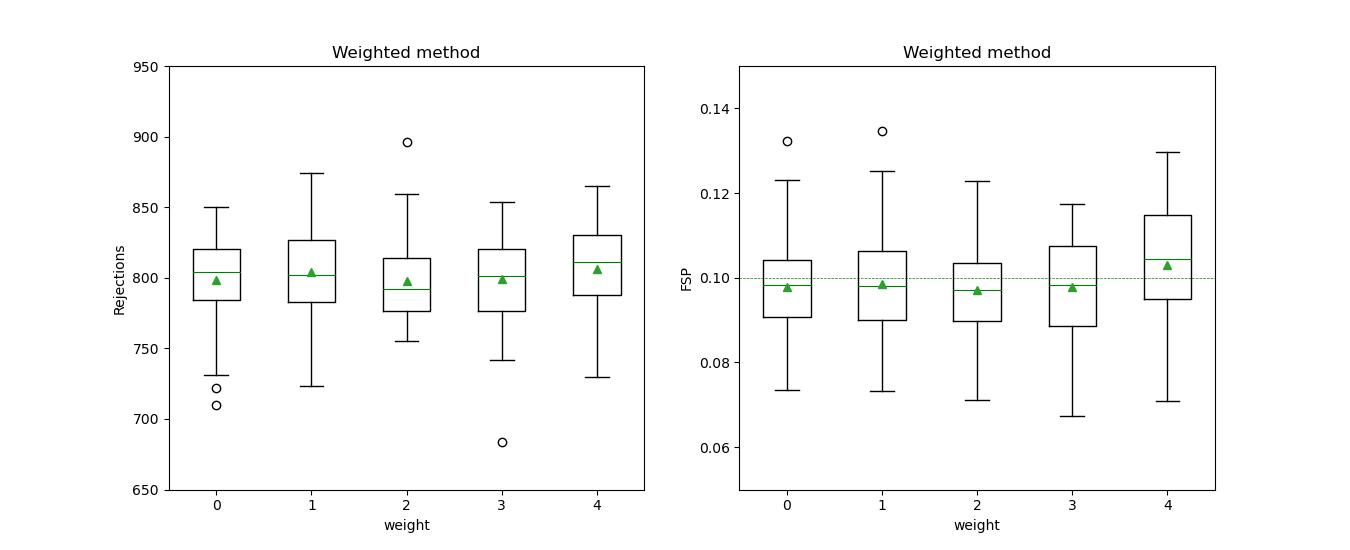}
\caption{We set the weights to 0, 1, 2, 3, and 4 respectively, and repeated the experiment one hundred times for each group. The picture on the left shows the change in the number of rejections with weight in our base method, and the image on the right shows the change of \mbox{FSP} with weight.}
\label{pic_weight}
\end{figure}
Finally, we set the weight $ K $ as $\{0, 1, 2, 3, 4\}$ respectively and repeated the experiment one hundred times for each group. We draw a box plot, and it is verified through the image that the weight method can control \mbox{FSR}, but the weight parameter currently does not show a strong correlation with \mbox{FSP}, as shown in fig \ref{pic_weight}.

\section{Discussions}\label{sectionsix}
\mbox{FSR} control has a wide range of applications in various fields of real life. By introducing indecision choices, we effectively reduce the error rate. Essentially, samples that are difficult to classify are handed over to experts for handling. We first introduce the optimality function in the oracle case and provide theoretical proof. We narrow down the search for the optimal solution of this problem to the case where the thresholds for all classes are identical. 

In the data-driven procedure, we provide the algorithm with exchangeability between $D^{cal}$ and $D^{test}$, the algorithm with different $\pi$, and the algorithm that allocates more weight to test data. We include proof of theoretical \mbox{FSR} control for all algorithms. Finally, we successfully extended binary classification to multi-classification tasks. 

However, when $\pi$ differs, the proposed algorithm still cannot accurately control the \mbox{FSR} at the threshold but only gives an unstable upper bound. In our subsequent research, we will continue to improve this algorithm by justifying the q-value formula. In the appendix, we proved the optimal decision function in multi-class classification by constructing a mapping, which is only suitable for multi-class classification problems. Therefore, we will continue to investigate these types of issues in future research.

\begin{appendix}
\section{Connection with multi-class classification method}\label{appA}
\subsection{Multi-class classification}
In the case of proving the optimality, we observe that the indecision method and the multi-class classification method mentioned in the previous section are equivalent to the problem of binary classification. By establishing a set of mappings, we can complete the unification of these two methods on the binary classification problem and prove the optimality of the multi-classification task on the binary classification.

We re-elaborate the ideas of the two methods: We have two classes, $\{1,2\}$. For the indecision-method, we have decision space$ \{\{ 1\},\{ 2\}, \{0\}\}$, and $$ \delta_{id}(t) =  {\arg\max}_{i} \, S^{i} \cdot\mathbb{I}(S\geq t).$$ When we do not have enough confidence to judge that $x$ belongs to a certain class, we make a prediction $\hat{Y} = 0 $. And for multi-class classification method, the decision space is $\{\{1\},\{2\},\{1,2\}\}$, and the decision rule is $$ \delta_{mc}(t)=\left\{
\begin{array}{rcl}
	\{1\},       &   \ when   & {S^{1}\geq t}.\\
	\{2\},     &    \ when & {S^{2} \geq t}.\\
	\{1,2\},       &   \ when   & {S < t, S^{1} + S^{2}\geq t}.
\end{array} \right. $$ where $ S = \max\{S^{1}, S^{2}\}$, the prediction$ \{1,2\}$ means that the probability of the real class of sample i in the prediction set is greater than the given threshold. For class c, the recall rate of the multi-class classification method is $ \mathbb{P}(c\in \hat{Y}|Y = c)$, and the global recall rate is $ \mathbb{P}(Y\in \hat{Y})$. We denote them as $r_{c}$ and $r$.

So if we make the following map $F: \ \delta_{id}(X) \rightarrow \delta_{mc}(X)$,
$$ F(\hat{Y}_{id})=\left\{
\begin{array}{rcl}
	\{1\},       &  \ when    & {\hat{Y} = 1}.\\
	\{2\},     &   \ when   & {\hat{Y} = 2}.\\
	\{1,2\},       &   \ when   & {\hat{Y} = 0}.
\end{array} \right. $$
We build a mapping that connects the indecision method with the multi-class classification method, and $F$ is a bijection.

We prove that if $S(x)$ is the optimal solution of the indecision method, then the decision rule for the multi-class classification method by mapping $F$ from $S$ is also the optimal solution. Assume for the indecision method $$ \mbox{mFSR} = \mathbb{P}(Y = \hat{Y}|\hat{Y} \neq 0) = \alpha,$$ and the number of test data is $ n^{test}$, and the number of indecisions is $ n^{id}$. So, for the multi-class classification method with bijection $F$, we have \begin{equation}\label{tran}
	\mathbb{P}(Y\in \hat{Y}) = \frac{(n^{test} - n^{id})\cdot \alpha}{n^{test}}.
\end{equation} From the conclusion of Section \ref{sectionthree}, we can know that under the decision rule $ \delta_{id} $, as the number of indecision increases, \mbox{mFSR} decreases. So, in equation \eqref{tran}, the error rate always varies monotonically with $\alpha$.

The power function in the first method is $$EPI = \frac{1}{n^{test}}\{\sum_{i \in D^{test}}\mathbb{I}(\hat{Y}_{i} = 0)\}.$$ We want to minimize the $ EPI$ under the control of \mbox{FSR}. In the second method, the power function is the average length of the prediction set, so we define $$ ||F(\hat{Y})||_{1}  = \frac{\sum_{i\in D^{test}}|F(\hat{Y}_{i})|}{n^{test}},$$ and we want to minimize it. We can also prove that the power of these two methods is equivalent.

\begin{equation}
	\begin{split}
		\min\{||F(\hat{Y})||_{1}\}  &= \min \left\{\frac{\sum_{i\in D^{test}}|F(\hat{Y}_{i})|}{n^{test}}\right\}  \\
		&= \min\left\{ \frac{n^{test} + \sum_{i\in D^{test}}\mathbb{I}(F(\hat{Y}_{i}) = \{1,2\})}{n^{test}}\right\} \\
		&= \min\left\{1 + \frac{\sum_{i\in D^{test}}\mathbb{I}(\hat{Y}_{i} = 0)}{n^{test}}\right\} \\
		&= \min\left\{\frac{\sum_{i\in D^{test}}\mathbb{I}(\hat{Y}_{i} = 0)}{n^{test}}\right\} \\
		&= \min\{EPI\}.
	\end{split}
\end{equation}
So, we can get the optimality of the second method on the binary classification problem from the optimality of the first method.
\begin{theorem}\label{multiclassoptimal}
	In the multi-class classification oracle situation on binary classification problem, for a given threshold $\beta$, let $ D_{\beta}$ denote the collection of selection rules that satisfy $ \mathbb{P}(Y\in \hat{Y})\geq \beta $. Let $||F(\hat{Y})||_{1}^{\delta}$ denote the value of $||F(\hat{Y})||_{1}$ of an arbitrary decision rule $\delta$. Then the oracle procedure is optimal in the sense that $||F(\hat{Y})||_{1}^{\delta_{mc}}\leq ||F(\hat{Y})||_{1}^{\delta} $ for any $ \delta \in D_{\beta}$ and finite sample size $n^{test}$.
\end{theorem}

\subsection{Optimality theory in selective K-class classification}
Assume that we have $K$ classes, and the score function is $S^{j}_{i} = P(Y_{i} = j| x_{i})$. We assume that the distribution of $S^{j}$ is
$$S^{j}(x) \sim F^{j}(x) = \pi_{1}\cdot F^{j}_{1} + \pi_{2}\cdot F^{j}_{2} + ... +\pi_{K}\cdot F^{j}_{K}, $$where $\pi_{i}$ is the proportion of class $i$ among all data, $F_{i}^{j}$ is the CDF of score function $S^{j}$ condition on class $i$. Inspired by simple situations, we will prove that the decision function $ S_{i} = \max\{S^{1}(x_{i}),S^{2}(x_{i}),..., S^{K}(x_{i})\}$ is optimal when the number of classes is $K$. Define the formula of $ Q(t)$ as \begin{equation} \label{eq1_K}
	\begin{split}
		Q(t)&=P(Y\neq \hat{Y}|S(x)\geq t),    \\
		&=\frac{E(\sum_{j\in D^{test}}((1-S_{j})\mathbb{I}(S_{j} \geq t)))}{E(\sum_{j\in D^{test}}\mathbb{I}(S_{j}\geq t))}.
	\end{split}
\end{equation} From Lemma \ref{lemma1}, the $\mbox{mFSR}(t)$ decreases monotonically with t. The decision rule can be changed to $$\delta^{i}_{OR} = \underset{j}{{\arg\max} \, S^{j}_{i}} \cdot\mathbb{I}\{S_{i} > Q^{-1}(\alpha)\}.$$

\begin{theorem}\label{optimaloracleKclass}
	Under the assumption of the distribution of $S^{1},S^{2},...,S^{K}$, and $\alpha$ have been chosen in advance. Let $ D_{\alpha}$ denote the collection of selection rules that satisfy $\mbox{mFSR} \leq \alpha$. Let $\mbox{ETS}_{\delta}$ denote the $\mbox{ETS}$ of an arbitrary decision rule $\delta$. Then, the oracle procedure is optimal in the sense that $\mbox{ETS}_{\delta_{OR}}\leq \mbox{ETS}_{\delta} $ for any $ \delta \in D_{\alpha}$. 
\end{theorem}

\subsection{Data-driven procedure in selective K-class classification}

We still assume that the groups $\{(X_{i}, Y_{i}): i \in D^{cal} \cup D^{test}\} $are exchangeable. Our decision rule is $$ \delta^{j}(t) =  \underset{i}{{\arg\max} \, \hat{S}^{i}_{j}} \cdot \mathbb{I}(\hat{S}_{j} \geq t).$$ And the estimated false discovery proportion(\mbox{FSP}), as a function of $t$, is given by
\begin{equation}\label{decision_K}
	\hat{Q}(t) = \frac{\frac{1}{n^{cal}+1}\{\sum_{i\in D^{cal}}\mathbb{I}(\hat{S_{i}}\geq t,\hat{Y}_{i} \neq Y_{i})+1\}}{\frac{1}{n^{test}}\{\sum_{j \in D^{test}}\mathbb{I}(\hat{S}_{j}\geq t)\}\bigvee 1},
\end{equation}
where $\hat{Y}_{i}$ denotes the predicted class of sample $i$, $Y_{i}$ denotes the real class of sample $i$, and $$\hat{S_{i}} = \max\{\hat{S}^{1}(X_{i}),\hat{S}^{2}(X_{i}),...,\hat{S}^{K}(x_{i})\},$$ where $\hat{S}^{j}$ is the score of class $j$ from the black-box model. We choose the smallest $t$ so the estimated \mbox{FSP} is less than $\alpha$. Define
\begin{equation}\label{tau_K}
	 \tau = \hat{Q}^{-1}(\alpha) = \inf \{t: \hat{Q}(t)\leq \alpha\},
\end{equation}
and \begin{equation}\label{R_value_K}
    \hat{R}_{j} = \inf_{t\leq \hat{s}_{j}}\{\hat{Q}(t)\},
\end{equation} where $ \hat{s}_{j} = \hat{S}(X_{j})$. Therefore $$ \mathbb{I}(\hat{R}_{j} \leq \alpha) \Leftrightarrow \mathbb{I}(\hat{S}_{j}\leq \tau).$$ So, the decision rule can be written as: for a given $\alpha$, $$\delta^{j} =  \underset{i}{{\arg\max} \, \hat{S}^{i}_{j}} \cdot \mathbb{I}(\hat{R}_{j} \leq \alpha).  $$ 

\renewcommand{\algorithmicrequire}{\textbf{Input:}}
\renewcommand{\algorithmicensure}{\textbf{Output:}}

\begin{algorithm}
\caption{\mbox{FSR} control procedure on multi-class classification.}
\label{alg:3}
\begin{algorithmic}[1]
\REQUIRE Existing data $D$ and its real class, test data $D^{test}$. 
\ENSURE Classification result of test data.
\STATE Randomly split $D$ into $D^{train}$ and $D^{cal}$. 
\STATE Train a black-box model only on $D^{train}$ to get score function $\{S^{1}(x)$, $S^{2}(x), ...,S^{K}(x)\}$.
\STATE Predict base scores for all observations in $D^{test}$ and $D^{cal}$.
\STATE Compute the q-value for all test data using equation \eqref{decision_K}.
\STATE Compute the R-value for all test data using equation \eqref{R_value_K}.
\STATE Threshold the R-value at a user-specified level $\alpha$, assigning an observation in $ D^{test}$ to class $\hat{Y}_{i}$ if $\hat{R}_{i}\leq \alpha$, where $$\hat{Y}_{i} =  \underset{j}{{\arg\max} \, S^{j}_{i}} \cdot\mathbb{I}(\hat{R}_{i}\leq\alpha). $$
\STATE Return an indecision result on all remaining observations where $\hat{R}_{i} \geq \alpha$.
\end{algorithmic}
\end{algorithm}

\begin{theorem}\label{data_driven_K}
	We denote $W^{test}$/$W^{cal}$ as the samples whose not real class corresponds to the maximum in $\{S^{1},S^{2},...,S^{K}\}$ in the test/calibration data, denote $|W^{test}|/n^{test}$ as $ p_{test}$, $|W^{cal}|/n^{cal}$ as $ p_{cal}$. Define $\gamma = \mathbb{E}(p_{test}/p_{cal})$, then under the exchangeability and the condition there exist $ \boldsymbol{x_{0}} = (x_{1},...,x_{K})$ such that the rank of function matrix $F(\boldsymbol{x_{0}})$ is K, the Algorithm \ref{alg:3} with R-value formula \eqref{decision_K} can control \mbox{FSR} at $ \alpha $ precisely, where
 $$
 F(\boldsymbol{x}) = \left(
 \begin{matrix}
    F^{1}_{1}( \boldsymbol{x}_{(1)}) & F^{1}_{2}(\boldsymbol{x}_{(2)}) & ... & F^{1}_{K}(\boldsymbol{x}_{(K)})\\  
    F^{2}_{1}(\boldsymbol{x}_{(1)}) & F^{2}_{2}(\boldsymbol{x}_{(2)}) & ... & F^{2}_{K}(\boldsymbol{x}_{(K)})\\
    ... & ... & ... & ...\\
    F^{K}_{1}(\boldsymbol{x}_{(1)}) & F^{K}_{2}(\boldsymbol{x}_{(2)}) & ... & F^{K}_{K}(\boldsymbol{x}_{(K)})
  \end{matrix}
  \right).
$$ 
\end{theorem}

\section{Proofs}\label{appB}

\begin{proof}[Proof of Lemma \ref{lemma1}]
	we denote $ \frac{p(x)}{\int_{t}^{1}p(x)dx}$ as $ P(x,t)$, therefore $ h(t) = \int_{t}^{1}P(x,t)dx = 1$, $ h(t)^{\prime} = 0$, so$$ h(t)^{\prime} = \int_{t}^{1}\frac{dP(x,t)}{dt}dx - P(t,t) = 0.$$
	And $ g(t) = \int_{t}^{1}P(x,t)f(x)dx$,
	\begin{align*}
			g(t)^{\prime} &= \int_{t}^{1}\frac{dP(x,t)}{dt}\cdot f(x)dx - P(t,t)f(t) \\
			&=\int_{t}^{1}\frac{dP(x,t)}{dt}\cdot f(x)dx - \int_{t}^{1}\frac{dP(x,t)}{dt}\cdot f(t)dx \\
			&=\int_{t}^{1}\frac{dP(x,t)}{dt}\cdot (f(x)-f(t))dx.
	\end{align*}
Because $f(x)$ is monotonically decreasing, $ f(x)-f(t) \leq 0 $ always holds. And \begin{equation}
		\frac{dP(x,t)}{dt} = -\frac{p(x)}{\{ \int_{t}^{1}p(x)dx\}^{2}}\cdot [-p(t)]
		= \frac{p(x)p(t)}{\{ \int_{t}^{1}p(x)dx\}^{2}} \geq 0.
\end{equation}
In summary, $g(t)^{\prime} \leq 0  $  is always established, so $g(t)$ is monotonically decreasing. With Assumption \ref{assumption1} holds, there is \begin{equation}
	\begin{split}
		 Q(t) = \frac{\mathbb{E}\{\sum_{j\in D^{test}}((1-S_{j})\mathbb{I}(S_{j} \geq t))\}}{\mathbb{E}(\sum_{j\in D^{test}}\mathbb{I}(S_{j}\geq t))} &= \frac{\sum_{j\in D^{test}}\int_{t}^{1}(1-x)p_{j}(x)dx}{\sum_{j\in D^{test}}\int_{t}^{1}p_{j}(x)dx} \\ &= \frac{\int_{t}^{1}(1-x)\sum_{j\in D^{test}}p_{j}(x)dx}{\int_{t}^{1}\sum_{j\in D^{test}}p_{j}(x)dx},
	\end{split}
\end{equation} where $ p_{j}(x)$ is the PDF of $S(x_{j})$, satisfy the condition of Lemma \ref*{lemma1}, so $Q(t)$ decreases monotonically. 
\begin{remark}
    The proof of Lemma \ref{lemma1} is from the perspective of function analysis, and it is noted that the proof process requires $p(x)$ to have continuity. But in fact, the density function of discrete random variables can only have countable first-type discontinuities, so the proof should also be true for discrete random variables.
\end{remark}

\end{proof}

\begin{proof}[Proof of Theorem \ref{optimaloracle}]
According to the definition, there is obviously $ P(Y\neq \hat{Y} | S\geq t) = P(Y\neq Y^{\prime} | S\geq t) $. where $$\hat{Y}_{j} = \underset{i}{{\arg\max} \, S_{j}^{i}} \cdot \mathbb{I}(S_{j} \geq t) , \ Y^{\prime}_{j} = \underset{i}{{\arg\max} \, S_{j}^{i}}.$$ So
\begin{align*}\label{eq_mfsr}
	    \mbox{mFSR}(t) &= P(Y\neq \hat{Y} | S\geq t) = P(Y\neq Y^{\prime} | S\geq t) \\ 
        &= \frac{P(Y \neq Y^{\prime}, S\geq t)}{P( S\geq t)} \\
		&= \frac{\mathbb{E} \{\sum_{i\in D^{test}}\mathbb{I}(Y_{i}\neq Y^{\prime}_{i}, S_{i}\geq t)\}}{\mathbb{E}\{\sum_{i\in D^{test}}\mathbb{I}(S_{i}\geq t)\}} \\
		&= \frac{\mathbb{E} \{\sum_{i\in D^{test}}\mathbb{I}(Y_{i}\neq Y^{\prime}_{i})\cdot \mathbb{I}(S_{i}\geq t)\}}{\mathbb{E}\{\sum_{i\in D^{test}}\mathbb{I}(S_{i}\geq t)\}} \\
		&= \frac{\mathbb{E}_{X} \{\sum_{i\in D^{test}} \mathbb{E}_{Y}\{\mathbb{I}(Y_{i}\neq Y^{\prime}_{i})\cdot \mathbb{I}(S_{i}\geq t) | x_i\}\}}{\mathbb{E}\{\sum_{i\in D^{test}}\mathbb{I}(S_{i}\geq t)\}}  \\
		&= \frac{\mathbb{E}_{X}\{  \sum_{i\in D^{test}} \mathbb{E}_{Y}\{\mathbb{I}(Y_{i}\neq Y^{\prime}_{i}) | x_i\} \cdot \mathbb{I}(S_{i}\geq t)\}} {\mathbb{E}\{\sum_{i\in D^{test}}\mathbb{I}(S_{i}\geq t)\}}  \\
		&= \frac{\mathbb{E} \{ \sum_{i\in D^{test}}(1-S_{i})\cdot \mathbb{I}(S_{i}\geq t)\}}{\mathbb{E}\{\sum_{i\in D^{test}}\mathbb{I}(S_{i}\geq t)\}}.
\end{align*}
From the formula above, we have:
$$\alpha \cdot \mathbb{E}\left\{\sum_{i\in D^{test}}\mathbb{I}(S_{i}\geq t)\} = \mathbb{E} \{ \sum_{i\in D^{test}}(1-S_{i})\cdot \mathbb{I}(S_{i}\geq t)\right\}.$$ 
Under the definition of $\delta$ and $S_{i} = \max\{S^{1}_{i},S^{2}_{i} \}$ \begin{equation}\label{lam0}
    \mathbb{E}\left \{\sum_{i\in D^{test}}(1-S_{i}-\alpha)\cdot\mathbb{I}(S_{i} \geq t)\right\} = 0,
\end{equation}
$$\mathbb{E} \left\{\sum_{i\in D^{test}}(1-S^{1}_{i}-\alpha)\cdot\mathbb{I}(\delta^{i} = 1) + (1-S^{2}_{i}-\alpha)\cdot\mathbb{I}(\delta^{i} = 2)\right\} = 0,$$ and by the equation $ S^{1} + S^{2} = 1 $,
\begin{equation}\label{eq3}
	\mathbb{E}\left\{\sum_{i \in D^{test}}(S^{2}_{i}-\alpha)\mathbb{I}(\delta^{i}_{OR} = 1) +(S^{1}_{i}-\alpha)\mathbb{I}(\delta^{i}_{OR}= 2)\right\} = 0.
\end{equation}
Let $\delta \in \{0,1,2\}^{m}$ be a general selection rule in $D_{\alpha}$. Then the \mbox{mFSR} constraints for $\delta$ imply that 
\begin{equation}\label{eq4}
	\mathbb{E}\left\{\sum_{i \in D^{test}}(S^{2}_{i}-\alpha)\mathbb{I}(\delta^{i} = 1) +(S^{1}_{i}-\alpha)\mathbb{I}(\delta^{i}= 2)\right\} \leq 0.
\end{equation}
The \mbox{ETS} of $\delta = \{\delta_{i}: n+1 \leq i \leq n+m\}$ is given by
\begin{equation}
	\mbox{ETS} = \mathbb{E} \left\{ \sum_{i \in D^{test}}S^{1}_{i}\mathbb{I}(\delta^{i} = 1) +S^{2}_{i}\mathbb{I}(\delta^{i}= 2) \right\}.
\end{equation}
According to equations \eqref{eq3} and \eqref{eq4}, we have :
\begin{equation}\label{eq6}
	\mathbb{E}\left\{\sum_{i \in D^{test}}(S^{2}_{i}-\alpha)\{\mathbb{I}(\delta^{i}_{OR} = 1)-\mathbb{I}(\delta^{i} = 1)\} +(S^{1}_{i}-\alpha)\{\mathbb{I}(\delta^{i}_{OR}= 2)-\mathbb{I}(\delta^{i}= 2)\}\right\} \geq 0.
\end{equation} 
Let $\lambda = (1-Q^{-1}(\alpha)-\alpha)/Q^{-1}(\alpha)$. From equation \eqref{lam0}, we have $ t \leq 1-\alpha$, then $ Q^{-1}(\alpha) \leq 1- \alpha $, $\lambda \geq 0 $. 
So the oracle rule can be written as $$\delta^{i}_{OR} =1\cdot \mathbb{I}\left\{\frac{S^{2}_{i}-\alpha}{S^{1}_{i}}<\lambda\right\} + 2 \cdot\mathbb{I}\left\{\frac{S^{1}_{i}-\alpha}{S^{2}_{i}}<\lambda\right\},$$ due to the monotonicity of the function$$ f(x) = \frac{x-\alpha}{1- x}.$$ Using the expression and techniques similar to the Neyman-Pearson Lemma, when $\mathbb{I}(\delta^{i}_{OR}) = 1$, then $S_{i}^{2} - \alpha -\lambda S_{i}^{1} < 0$ and $\mathbb{I}(\delta^{i}_{OR}) - \mathbb{I}(\delta^{i}) \geq 0 $, and vice versa.  So we claim that the following result holds:
$$\{\mathbb{I}(\delta^{i}_{OR}= 1)-\mathbb{I}(\delta^{i}= 1)\}\cdot(S^{2}_{i}-\alpha -\lambda S^{1}_{i}) \leq 0,$$
and
$$\{\mathbb{I}(\delta^{i}_{OR}= 2)-\mathbb{I}(\delta^{i}= 2)\}\cdot(S^{1}_{i}-\alpha- \lambda S^{2}_{i}) \leq 0.$$
It follows that:
\begin{equation}\label{eq7}
	\mathbb{E}\left\{\sum_{i\in D^{test}}[\{\mathbb{I}(\delta^{i}_{OR}= 1)-\mathbb{I}(\delta^{i}= 1)\}(S^{2}_{i}-\alpha-\lambda S^{1}_{i}) + \{\mathbb{I}(\delta^{i}_{OR}= 2)-\mathbb{I}(\delta^{i}= 2)\}(S^{1}_{i}-\alpha-\lambda S^{2}_{i})] \right\} \leq 0.
\end{equation}
According to equation \eqref{eq6} and \eqref{eq7}, we have \begin{equation}
	\begin{split}
		&\lambda_{OR}(\mbox{ETS}_{\delta_{OR}}- \mbox{ETS}_{\delta}) \\
		&= \lambda_{OR} \sum_{i \in D^{test}}\{\mathbb{I}(\delta^{i}_{OR}= 1)-\mathbb{I}(\delta^{i}= 1)\}S^{1}_{i}+ \{\mathbb{I}(\delta^{i}_{OR}= 2)-\mathbb{I}(\delta^{i}= 2)\}S^{2}_{i}\geq 0.
	\end{split}
\end{equation}The proof is complete.
\end{proof}

\begin{proof}[Proof of Theorem \ref{data_driven}]
Let $$ V^{test}(t) = \sum_{j\in D^{test}}\mathbb{I}(\hat{S}_{j}\geq t,Y_{j}\neq \hat{Y}_{j}), \ R^{test}(t) = \sum_{j\in D^{test}}\mathbb{I}(\hat{S}_{j}\geq t),$$and corresponding quantities in $ D^{cal}$ $$ V^{cal}(t) = \sum_{j\in D^{cal}}\mathbb{I}(\hat{S}_{j}\geq t, \ Y_{j}\neq \hat{Y}_{j}), R^{cal}(t) = \sum_{j\in D^{cal}}\mathbb{I}(\hat{S}_{j}\geq t).$$ The \mbox{FSP} of the proposed algorithm is given by$$ \mbox{FSP}(\tau) = \frac{V^{test}(\tau)}{R^{test}(\tau)\bigvee 1}.$$We expand the expression above

\begin{equation}
	\begin{split}
		\mbox{FSP}(\tau) &= \frac{V^{test}(\tau)}{V^{cal}(\tau)+1}\cdot \frac{V^{cal}(\tau)+1}{R^{test}(\tau)\bigvee 1} \\
		&= \hat{Q}(\tau)\cdot \frac{n^{cal}+1}{n^{test}} \cdot \frac{V^{test}(\tau)}{V^{cal}(\tau)+1} \\
		&\leq \alpha \cdot \frac{n^{cal}+1}{n^{test}} \cdot \frac{V^{test}(\tau)}{V^{cal}(\tau)+1}.
	\end{split}
\end{equation}

A key step to establish the \mbox{FSR} control, i.e. $\mathbb{E}\{\mbox{FSP}(\tau)\}\leq \alpha$, is to show that the ratio
\begin{equation}\label{mart}
	\frac{V^{test}(t)}{V^{cal}(t) + 1},
\end{equation}
is a backward martingale. In our proof, we force the following discrete-time filtration that describes the misclassification process:$$ F_{k} = \{ \Delta(V^{test}(s_{k}),V^{cal}(s_{k}))\}_{t_{l}\leq s_{k}\leq t},$$
where $s_{k}$ corresponds to the threshold(time) when exactly k subjects, combining the subjects in both $ D^{cal}$ and $ D^{test}$, are mistakenly classified, whose real class corresponds to a lower score in $\{S^{1}, S^{2}\}$. So $ V^{test}(t_{l})$ is $ W^{test}$, and $ V^{cal}(t_{l})$ is $ W^{cal}$. Note that at time $s_{k}$, only one of the two following events is possible.

\begin{equation}
	\begin{split}
		A_{1} = \mathbb{I}\{V^{test}(s_{k-1}) = V^{test}(s_{k}), and \ V^{cal}(s_{k-1}) = V^{cal}(s_{k}) - 1\},
		\\
		A_{2} = \mathbb{I}\{V^{test}(s_{k-1}) = V^{test}(s_{k}) -1, and\ V^{cal}(s_{k-1}) = V^{cal}(s_{k})\}.
	\end{split}
\end{equation}
According to Assumption \ref{assumption1}, which claims that  test data and calibration data are exchangeable, and the fact that FASI uses the same fitted model to compute the scores, we have$$ \mathbb{P}(A_{1}|F_{k}) = \frac{V^{cal}(s_{k})}{V^{test}(s_{k}) + V^{cal}(s_{k})}; \ \mathbb{P}(A_{2}|F_{k}) =  \frac{V^{test}(s_{k})}{V^{test}(s_{k}) + V^{cal}(s_{k})},$$
and note that:
\begin{equation}
	\begin{split}
		&\mathbb{E}\{\frac{V^{test}(s_{k-1})}{V^{cal}(s_{k-1}) + 1}|F_{k}\} \\
		&=\frac{V^{test}(s_{k})}{V^{cal}(s_{k})}\cdot \frac{V^{cal}(s_{k})}{V^{test}(s_{k}) + V^{cal}(s_{k})} + \frac{V^{test}(s_{k})-1}{V^{cal}(s_{k})+1} \cdot \frac{V^{test}(s_{k})}{V^{test}(s_{k}) + V^{cal}(s_{k})}\\
		&=\frac{V^{test}(s_{k})}{V^{cal}(s_{k}) + 1}.
	\end{split}
\end{equation}
Therefore, the expression \eqref{mart} is a martingale. The threshold $\tau$ defined by expression \eqref{tau} is a stopping time with respect to the filtration$\ F_{k} \ $since $ \{\tau \leq s_{k}\} \in F_{k}$. In other words, whether the $k$th misclassification occurs completely depends on the information prior to time $s_{k}$ (including $s_{k}$). And we make the following statement.

\begin{itemize}
	\item When $ t_{l}$ is used, all subjects are classified under the rule $\delta = {\arg\max}_{i} \, S^{i}.$
	\item The sizes of the testing and calibration sets are random. The expectation is taken in steps, first conditional on fixed sample sizes and then taken over all possible sample sizes.
\end{itemize}
From the definition of $ p_{test} $, we know:
\begin{equation}
	\mathbb{E}(p_{test}) = \mathbb{E}(|W^{test}|/n^{test}) = \mathbb{E}\left\{ \frac{\sum_{i=1}^{n^{test}}\{\mathbb{I}(S^{1}(x_{i})\geq S^{2}(x_{i}),c = 2) + \mathbb{I}(S^{1}(x_{i}) < S^{2}(x_{i}),c = 1)\}}{n^{test}} \right\}.
\end{equation}
In summary:
\begin{equation}
	\begin{split}
		\mbox{FSR} &= \mathbb{E}\{\mbox{FSP}(\tau)\} \\
		&\leq \alpha \cdot \mathbb{E}[\frac{n^{cal} + 1}{n^{test}}\cdot  \mathbb{E}\{\frac{V^{test}(\tau)}{V^{cal}(\tau) + 1}| D^{cal},D^{test}\}] \\ 
		&=\alpha\cdot \mathbb{E}[\frac{n^{cal} + 1}{n^{test}}\cdot  \mathbb{E}\{\frac{V^{test}(t_{l})}{V^{cal}(t_{l}) + 1}| D^{cal},D^{test}\}] \\ 
		&=\alpha\cdot \mathbb{E}[\frac{n^{cal} + 1}{n^{test}}\cdot \frac{| W^{test}|}{| W^{cal}| + 1}] \\ 
		&=\alpha\cdot \mathbb{E}[\frac{n^{cal} + 1}{| W^{cal}| + 1}]\cdot \mathbb{E}[ \frac{| W^{test}|}{n^{test}}] \\
		&\leq\alpha\cdot\mathbb{E}\{\frac{p_{test}}{p_{cal}}\} \coloneqq \gamma\alpha.
	\end{split}
\end{equation}
Here, we control the \mbox{FSR} at $\alpha\cdot\mathbb{E}\{p_{test}/p_{cal}\}$. From equation \eqref{p_test} and the data are exchangeable, treating $n^{test}$ as a fixed constant, and according to the law of large numbers, the estimation of $ p_{test}$ converges to the following formula with probability 1.
\begin{equation}
	\begin{split}
		\mathbb{E}(p_{test}) = &\mathbb{P}(S^{1}(x_{i})\geq 0.5, c = 2) + \mathbb{P}(S^{1}(x_{i})< 0.5, c = 1) \\
		= & (1-\pi_{test})\cdot(1 - F^{1}_{2}(0.5)) + \pi_{test} \cdot F^{1}_{1}(0.5).
	\end{split}
\end{equation}So as the $ p_{cal}$. Notice that when test data and calibration data are exchangeable, the distribution of test data and calibration are the same. From the condition there exist $x_{0}, x_{1}$, s.t. $F_{1}^{1}(x_{0})\neq F_{1}^{2}(x_{0}) $, $ F_{1}^{1}(x_{1})\neq F_{2}^{1}(x_{1}) $, we can conclude that $\pi_{test} = \pi_{cal} $, so $ p_{cal} = p_{test}$,  $\gamma = 1$. The proof is complete.
\end{proof}

\begin{proof}[Proof of Lemma \ref{lowerbound}]
Notice that $$\mathbb{E}(p^{test}) =  1 -  F^{1}_{2}(0.5) + \pi_{test}\cdot( F^{1}_{2}(0.5) - 1 + F^{1}_{1}(0.5)),$$
and Storey's estimator result in $\hat{\pi}_{test}(\lambda) \geq \pi_{test}$ for any $\lambda$. From the definition of $\hat{\pi}$, when $$\ F^{1}_{2}(0.5) + F^{1}_{1}(0.5) -1 \geq 0,$$ $\hat{\pi}(\lambda) \geq \pi $ and when $$ \ F^{1}_{2}(0.5) + F^{1}_{1}(0.5) -1 < 0,$$ $\hat{\pi}(\lambda) \leq \pi $ for any $\lambda$. With Assumption \ref{assumption1} hold, \begin{align*}
        \mathbb{E}\{ \hat{p}_{test}\} = & \mathbb{E}\{ 1 -  \hat{F}^{1}_{2}(0.5) + \hat{\pi}_{test}\cdot( \hat{F}^{1}_{2}(0.5) - 1 + \hat{F}^{1}_{1}(0.5))\} \\
        = & \mathbb{E}_{D^{test}} \{\mathbb{E}_{D^{cal}} \{1 -  \hat{F}^{1}_{2}(0.5) + \hat{\pi}_{test}\cdot( \hat{F}^{1}_{2}(0.5) - 1 + \hat{F}^{1}_{1}(0.5))\}\} \\
        = & \mathbb{E}_{D^{test}} \{ 1 -  F^{1}_{2}(0.5) + \hat{\pi}_{test}\cdot( F^{1}_{2}(0.5) - 1 + F^{1}_{1}(0.5))\} \\
        = & 1 -  F^{1}_{2}(0.5) +  \mathbb{E}_{D^{test}}\{\hat{\pi}_{test}\}\cdot( F^{1}_{2}(0.5) - 1 + F^{1}_{1}(0.5)) \\
        \geq & 1 -  F^{1}_{2}(0.5) +  \pi_{test} \cdot( F^{1}_{2}(0.5) - 1 + F^{1}_{1}(0.5)) \\
        = & \mathbb{E}(p_{test}).
\end{align*}
\end{proof}

\begin{proof}[Proof of Lemma \ref{exchangeable}]
    The first conclusion is from article \cite{ref23}, and we will prove the second conclusion. From exchangeability, $\mathbb{E}[X|\mathcal{G}]$, where its $s_{k}$ is independent of $X$ is equal to $\mathbb{E}[X|\mathcal{G}]$, where its $s_{n}$ is independent of $X$.
    \begin{equation}
        \begin{split}
            \mathbb{E}[X|\mathcal{G}] = &\mathbb{E}[X|\sigma(s_k,s_{k+1},...,s_n)] \\
            =& \mathbb{E}[X|\sigma(s_k,s_{k+1},...,s_{n-1}), s_{n-1}\leq s_{n}] \\
            =& \mathbb{E}[X|\sigma(s_k,s_{k+1},...,s_{n-1})] \\ 
            =&  \mathbb{E}(X|\mathcal{H}).
        \end{split}
    \end{equation}
\end{proof}

\begin{proof}[Proof of Lemma \ref{chengji}]
    \begin{equation}
        \begin{split}
            & | \mathbb{E}\{X_{n}\cdot Y_{n} | \sigma(s_{j})_{j\leq n-1} \} - X_{n-1}\cdot Y_{n-1} | \\
            = &| \mathbb{E}\{X_{n}\cdot Y_{n}- X_{n-1}\cdot Y_{n-1}  | \sigma(s_{j})_{j\leq n-1} \} |\\
            = & |\mathbb{E}\{ (X_{n} - X_{n-1}) \cdot Y_{n} + X_{n-1}\cdot (Y_{n} - Y_{n-1}) | \sigma(s_{j})_{j\leq n-1} \} |\\
            \leq & |\mathbb{E}\{ (X_{n} - X_{n-1}) \cdot Y_{n} | \sigma(s_{j})_{j\leq n-1} \} | + |\mathbb{E}\{ X_{n-1}\cdot (Y_{n} - Y_{n-1}) | \sigma(s_{j})_{j\leq n-1} \} |\\
            \leq & M \cdot |\mathbb{E}\{ X_{n} - X_{n-1} | \sigma(s_{j})_{j\leq n-1}\}| + M\cdot | \mathbb{E}\{ Y_{n} - Y_{n-1} | \sigma(s_{j})_{j\leq n-1} \} | \\
            = & 0.
        \end{split}
    \end{equation}Therefore $(X_{n}\cdot Y_{n}, \sigma(s_{j})_{ j\leq n}, n\geq 0)$ is also a martingale.
\end{proof}

\begin{proof}[Proof of Lemma \ref{piestimate}]
    For process \begin{equation}\label{martingaleW}
        \frac{V_{1}^{test}(\tilde \tau)}{V_{1}^{cal}(\tilde \tau) + 1},
    \end{equation}only one of the two following events are possible\begin{equation}
    \begin{split}
        A_{1} &= \mathbb{I}\{V_{1}^{test}(r^{1}_{k-1}) = V^{test}(r^{1}_{k}) - 1, \ and \ V_{1}^{cal}(r^{1}_{k-1}) = V_{1}^{cal}(r^{1}_{k})   \  \},
        \\
        A_{2} &= \mathbb{I}\{V_{1}^{test}(r^{1}_{k-1}) = V_{1}^{test}(r^{1}_{k}), \ and \ V_{1}^{cal}(r^{1}_{k-1}) = V_{1}^{cal}(r^{1}_{k}) - 1 \},
    \end{split}
\end{equation} where $r^{1}_{k}$ corresponds to the threshold (time) when exactly k subjects in $D^{test}$ and $D^{cal}$ belong to class 1 but are misclassified to class 2. From the Assumption \ref{assumption3}, we have$$ \mathbb{P}(A_{1}|F_{k}) = \frac{V^{test}_{1}(r^{1}_{k})}{V_{1}^{test}(r^{1}_{k}) + V_{1}^{cal}(r^{1}_{k})}; \ \mathbb{P}(A_{2}|F_{k}) =  \frac{V^{cal}_{1}(r^{1}_{k})}{V_{1}^{test}(r^{1}_{k}) + V_{1}^{cal}(r^{1}_{k})}.$$
So \begin{equation}
	\begin{split}
		&\mathbb{E}\{\frac{V_{1}^{test}(r^{1}_{k-1})}{V_{1}^{cal}(r^{1}_{k-1}) + 1}|F_{k}\} \\
		& = \frac{V_{1}^{test}(r^{1}_{k}) - 1 }{V_{1}^{cal}(r^{1}_{k}) + 1} \cdot \frac{V^{test}_{1}(r^{1}_{k})}{V_{1}^{test}(r^{1}_{k}) + V_{1}^{cal}(r^{1}_{k})} + \frac{V^{test}_{1}(r^{1}_{k}) }{V_{1}^{cal}(r^{1}_{k})} \cdot \frac{V^{cal}_{1}(r^{1}_{k})}{V_{1}^{test}(r^{1}_{k}) + V_{1}^{cal}(r^{1}_{k})} \\
            &=\frac{V_{1}^{test}(r^{1}_{k})}{V_{1}^{cal}(r^{1}_{k}) + 1}.
	\end{split}
\end{equation}
Therefore, process \eqref{martingaleW} is a martingale. And process\begin{equation}\label{martingaleW1}
    \frac{V^{test}_{2}(\tilde \tau)  }{ V^{cal}_{2}(\tilde \tau) + 1} 
\end{equation} is also a martingale, and the proof is similar to the above.

\begin{itemize}
    \item The filtration of martingale \eqref{martingaleW} is $ F^{1}_{k} = \{ \sigma (s_{j})_{k \leq j \leq n} \}$ where $s_{k}$ corresponds to the threshold (time) when exactly k subjects, combining the subjects in  $D^{test}$ and $D^{cal}$, belong to class 2 but are misclassified to class 1.
    \item The filtration of martingale \eqref{martingaleW1} is $ F^{2}_{k} = \{ \sigma (s_{j})_{k \leq j \leq n} \}$ where $s_{k}$ corresponds to the threshold (time) when exactly k subjects, combining the subjects in  $D^{test} \cup D^{cal}$, which belong to class 1 but are misclassified to class 2. 
\end{itemize}
By Lemma \ref{exchangeable}, Lemma \ref{chengji}, and the dependence between  class 1 and class 2, the proof is completed.

\end{proof}

\begin{proof}[Proof of Theorem \ref{datadrivenpidifferent}]
The binary classification \mbox{FSR} can be work out:

\begin{align*}
        \mbox{FSR} &= \mathbb{E}(\mbox{FSP}) \\
    &= \mathbb{E} \left\{ \tilde Q(\tilde \tau)\cdot \frac{n^{cal}+1}{n^{test}} \cdot \frac{V^{test}(\tilde \tau)}{V^{cal}(\tilde \tau)+1} \right\} \\
    &\leq \alpha\cdot \mathbb{E} \left\{ \frac{n^{cal}+1}{n^{test}} \cdot \frac{\hat{p}_{cal}}{\hat{p}_{test}} \cdot\frac{V^{test}(\tilde \tau)}{V^{cal}(\tilde \tau)+1} \right\}   \\
    &\leq \alpha\cdot \mathbb{E} \left\{ \frac{n^{cal}+1}{n^{test}}\cdot \frac{\hat{p}_{cal}}{\hat{p}_{test}} \cdot \left( \frac{V^{test}_{1}(\tilde \tau)}{V^{cal}_{1}(\tilde \tau)+1} + \frac{V^{test}_{2}(\tilde \tau)}{V^{cal}_{2}(\tilde \tau)+1}\right)\right\}   \\
    &= \alpha\cdot \mathbb{E} \left\{ \frac{n^{cal}+1}{n^{test}}\cdot \frac{\hat{p}_{cal}}{\hat{p}_{test}} \cdot \mathbb{E}\left\{ \frac{V^{test}_{1}(t_{l})}{V^{cal}_{1}(t_{l})+1} + \frac{V^{test}_{2}(t_{l})}{V^{cal}_{2}(t_{l})+1}|D^{cal}, D^{test}\right\}\right\} \\
    &= \alpha\cdot \mathbb{E} \left\{ \frac{n^{cal}+1}{n^{test}}\cdot \frac{\hat{p}_{cal}}{\hat{p}_{test}} \cdot  \left(\frac{W^{test}_{1}}{W^{cal}_{1}+1} + \frac{W^{test}_{2}}{W^{cal}_{2}+1}\right) \right\} \\
    &\leq \alpha\cdot \mathbb{E}\left\{ \frac{\hat{p}_{cal}}{\hat{p}_{test}} \cdot \left( \frac{W^{test}_1}{n^{test}}\left/\frac{W^{cal}_1}{n^{cal}} \right) + \frac{\hat{p}_{cal}}{\hat{p}_{test}} \cdot\left( \frac{W^{test}_2}{n^{test}}\right/\frac{W^{cal}_2}{n^{cal}}\right)   \right\} \\
    &\leq \alpha\cdot \frac{\mathbb{E}(p_{cal})}{\mathbb{E}(p_{test})} \cdot\left( \frac{\pi_{test}}{\pi_{cal}} + \frac{1 - \pi_{test}}{1-\pi_{cal}}\right),
\end{align*}
where $W_{i}^{test}$ is the number of samples classified to class 1 with no indecision choices.
\begin{remark}
    When $\pi$ is extremely large or small, this inequality gives a large upper bound. So, we reconsider the exchangeability within each class and give a new q-value formula:
\begin{equation}\label{decision_1}
	\hat{Q}_{1}(t) = \frac{\frac{1}{n^{cal}+1}\{\sum_{i\in D^{cal}}\mathbb{I}(Y_{i} = 2, \hat{Y}_{i} = 1)+1\}}{\frac{1}{n^{test}}\{\sum_{n+j \in D^{test}}\mathbb{I}(\hat{Y}_{i} = 1)\}\bigvee 1},
\end{equation}
\begin{equation}\label{decision_2}
	\hat{Q}_{2}(t) = \frac{\frac{1}{n^{cal}+1}\{\sum_{i\in D^{cal}}\mathbb{I}(Y_{i} = 1, \hat{Y}_{i} = 2)+1\}}{\frac{1}{n^{test}}\{\sum_{n+j \in D^{test}}\mathbb{I}(\hat{Y}_{i} = 2)\}\bigvee 1},
\end{equation}
where $\hat{Y}_{i}$ denotes the predicted class of sample $i$, $Y_{i}$ denotes the real class of sample $i$. The decision rule is  just like FASI $$ \delta^{j}(t_{1}, t_{2}) =  1 \cdot \mathbb{I}(\hat{S}^{1}_{j} \geq t_{1}) + 2 \cdot \mathbb{I}(\hat{S}^{2}_{j} \geq t_{2}).$$ We choose the smallest $t$ so the estimated \mbox{FSP} is less than $\alpha$. Define
\begin{equation}
	 \tau_{i} = \hat{Q}_{i}^{-1}(\alpha) = \inf \{t: \hat{Q}_{i}(t)\leq \alpha_{i}\},
\end{equation}
where $ \alpha_{1} = \alpha\cdot \pi_{cal}$ and $ \alpha_{2} = \alpha\cdot (1-\pi_{cal})$. 
So, equivalently, we can change the q-value formula to 
\begin{equation}
    \begin{split}
        \tilde{Q}_{1}(t) &= \frac{\frac{1}{n^{cal}+1}\{\sum_{i\in D^{cal}}\mathbb{I}(Y_{i} = 2, \hat{Y}_{i} = 1)+1\}}{\{\frac{1}{n^{test}}(\sum_{n+j \in D^{test}}\mathbb{I}(\hat{Y}_{i} = 1))\bigvee 1\} \cdot \hat{\pi}_{1}^{cal}} \\
        &= \frac{\{\sum_{i\in D^{cal}}\mathbb{I}(Y_{i} = 2, \hat{Y}_{i} = 1)+1\}}{\{\frac{1}{n^{test}}(\sum_{n+j \in D^{test}}\mathbb{I}(\hat{Y}_{i} = 1))\bigvee 1\} \cdot (\sum_{i \in D^{cal}}\mathbb{I}(Y_{i} = 1))},
    \end{split}
\end{equation}
where $ \hat{\pi}_{i}^{cal} $ is the proportion of class i in test data. Define \begin{equation}
    \hat{R}^{i}_{n+j} = \inf_{t\leq \hat{s}^{i}}\{\hat{Q}_{i}(t)\} \ \ and \ \  \tilde{R}^{i}_{n+j} = \inf_{t\leq \hat{s}^{i}}\{\tilde{Q}_{i}(t)\},
\end{equation} where $ \hat{s}^{i} = \hat{S}^{i}(X_{n+j}), \ i = 1,2 $. Therefore $$ \mathbb{I}(\tilde{R}^{i}_{n+j} \leq \alpha) \Leftrightarrow  \mathbb{I}(\hat{R}^{i}_{n+j} \leq \alpha_{i}) \Leftrightarrow \mathbb{I}(\hat{S}^{i}_{n+j}\leq \tau_{i}).$$
The \mbox{FSR} can be work out:
\begin{align*}
        \mbox{FSR} &= \mathbb{E}(\mbox{FSP}) \\
        &= \mathbb{E}\left\{ \frac{V^{test}(\tau)}{R^{test}(\tau)} \right\} \\
        &= \mathbb{E}\left\{ \frac{V_{1}^{test}(\tau)}{R^{test}(\tau)} + \frac{V_{2}^{test}(\tau)}{R^{test}(\tau)} \right\} \\
        &= \mathbb{E}\left\{ \frac{V_{1}^{cal}(\tau) + 1}{R^{test}(\tau)}\cdot \frac{V_{1}^{test}(\tau)}{V_{1}^{cal}(\tau) + 1}+ \frac{V_{2}^{cal}(\tau) + 1}{R^{test}(\tau)}\cdot \frac{V_{2}^{test}(\tau)}{V_{2}^{cal}(\tau) + 1} \right\} \\
        &= \mathbb{E}\left\{ \frac{V_{1}^{cal}(\tau) + 1}{R^{test}(\tau)}\cdot \frac{V_{1}^{test}(t_{l})}{V_{1}^{cal}(t_{l}) + 1} + \frac{V_{2}^{cal}(\tau) + 1}{R^{test}(\tau)}\cdot \frac{V_{2}^{test}(t_{l})}{V_{2}^{cal}(t_{l}) + 1} \right\} \\
        & \leq \mathbb{E}\left\{ \alpha_{1}\cdot \frac{n^{cal} + 1}{n^{test}}\cdot \frac{V_{1}^{test}(t_{l})}{V_{1}^{cal}(t_{l}) + 1} + \alpha_{2}\cdot \frac{n^{cal} + 1}{n^{test}}\cdot \frac{V_{2}^{test}(t_{l})}{V_{2}^{cal}(t_{l}) + 1} \right\}\\
        & \leq  \alpha_{1} \cdot  \frac{\mathbb{P}_{test}(\hat{Y} = 1,Y = 2)}{\mathbb{P}_{cal}(\hat{Y} = 1,Y = 2) } + \alpha_{2}\cdot \frac{\mathbb{P}_{test}(\hat{Y} = 2, Y = 1)}{\mathbb{P}_{cal}(\hat{Y} = 2, Y = 1) } \\
        & = \alpha_{1}\cdot \frac{\mathbb{P}_{test}(\hat{Y} = 1 | Y = 2) \cdot \mathbb{P}_{test}(Y = 2)}{\mathbb{P}_{cal}(\hat{Y} = 1 | Y = 2) \cdot \mathbb{P}_{test}(Y = 2) } + \alpha_{2} \cdot \frac{\mathbb{P}_{test}(\hat{Y} = 2| Y = 1) \cdot \mathbb{P}_{test}(Y = 1)}{ \mathbb{P}_{cal}(\hat{Y} = 2| Y = 1) \cdot \mathbb{P}_{test}(Y = 1) } \\
        & = \alpha_{1}\cdot \frac{1 - \pi_{test}}{1 - \pi_{cal}} + \alpha_{2} \cdot \frac{\pi_{test}}{\pi_{cal}} \\ 
        &=\pi_{test}\cdot \alpha + (1-\pi_{test})\cdot \alpha = \alpha
\end{align*}
However, setting different thresholds for different classes goes against our original intention, so here, we only give a way to control \mbox{FSR} for different $\pi$ without going too deep. This method is essentially conservative, and from the previous proof of the oracle situation, we know this data-driven procedure is not optimal. In the following research work, we will continue to improve our q-value method to maximize power as much as possible.
\end{remark}

\end{proof}

\begin{proof}[Proof of Corollary \ref{weightmartingale}]
	In our proof, note that at time $s_{k}$, only one of the two following events is possible:
	\begin{equation}
		\begin{split}
			A_{1} = \mathbb{I}\{R^{test}(s_{k-1}) = R^{test}(s_{k}), and \ R^{cal}(s_{k-1}) = R^{cal}(s_{k}) - 1\},
			\\
			A_{2} = \mathbb{I}\{R^{test}(s_{k-1}) = R^{test}(s_{k}) -1, and\ R^{cal}(s_{k-1}) = R^{cal}(s_{k})\}.
		\end{split}
	\end{equation}
	According to the assumption claims that test data and calibration data are exchangeable and the fact that FASI uses the same fitted model to compute the scores, we have$$ \mathbb{P}(A_{1}|F_{k}) = \frac{R^{cal}(s_{k})}{R^{test}(s_{k}) + R^{cal}(s_{k})}; \ \mathbb{P}(A_{2}|F_{k}) =  \frac{R^{test}(s_{k})}{R^{test}(s_{k}) + R^{cal}(s_{k})},$$
	and note that:
	\begin{align*}
			&\mathbb{E}\{\frac{R^{cal}(s_{k-1})+K\cdot R^{test}(s_{k-1}) + K}{R^{test}(s_{k-1} )+ 1 }|F_{k}\} \\
			&=\frac{R^{cal}(s_{k})+K\cdot R^{test}(s_{k}) + K- 1}{R^{test}(s_{k} )+ 1 } \cdot \frac{R^{cal}(s_{k})}{R^{test}(s_{k}) + R^{cal}(s_{k})} +  \frac{R^{cal}(s_{k})+K\cdot R^{test}(s_{k})}{R^{test}(s_{k} ) }\\ &\cdot \frac{R^{test}(s_{k})}{R^{test}(s_{k}) + R^{cal}(s_{k})} \\
			&=\frac{R^{cal}(s_{k})+K \cdot R^{test}(s_{k}) + K}{R^{test}(s_{k} )+ 1 }.
	\end{align*}
	Therefore, expression 3.9 is a martingale.
\end{proof}

\begin{proof}[Proof of Theorem \ref{datadrivenweighted}]
	From the definition of $ Q(t)$ above, we can know that
	$$\begin{aligned}
			\mbox{FSP}(\tau) & =\frac{V^{test}(\tau)}{V^{cal}(\tau)+1} \cdot \frac{V^{cal}(\tau)+1}{R^{cal}(\tau)+K \cdot R^{test}(\tau)+K} \cdot \frac{R^{cal}(\tau)+K \cdot R^{test}(\tau)+K}{ R^{test}(\tau)+1} \\
			& \leq \alpha \cdot \frac{n^{cal}+1}{n^{cal}+ K \cdot n^{\text {test}}+ K } \cdot \frac{V^{test}(\tau)}{V^{cal}(\tau)+1} \cdot \frac{R^{cal}(\tau)+K \cdot R^{test}(\tau)+K}{R^{test}(\tau)+1}.
		\end{aligned}$$
	Therefore
$$\begin{aligned}
	\mbox{FSR} & =\mathbb{E}\left\{\mbox{FSP}(\tau)\right\} \\
	& \leq \alpha \mathbb{E}\left[\frac{\left|D^{cal}\right|+1}{\left|D^{cal}\right|+\left|D^{test}\right|+1} \cdot \frac{R^{cal}(\tau)+K \cdot R^{test}(\tau)+ K }{R^{test}(\tau)+1} \cdot \mathbb{E}\left\{\frac{V^{test}(\tau)}{V^{cal}(\tau)+1} \mid D^{cal}, D^{test}\right\}\right].
\end{aligned}$$
From the prove of Theorem \ref{data_driven}, we have:
$$
\mathbb{E}\left\{\frac{V^{test}(\tau)}{V^{cal}(\tau)+1} \mid D^{cal}, D^{test}\right\}=\frac{V^{test}\left(t_l\right)}{V^{cal}(t)+1}=\frac{\left|W^{test}\right|}{\left|W^{cal}\right|+1},
$$ 

$$
\begin{aligned}
	\mathbb{E}\left\{\frac{R^{cal}(\tau)+K\cdot R^{test}(\tau)+K}{R^{test}(\tau)+1}\right\} & =\mathbb{E}\left\{\frac{R^{cal}\left(t_l\right)+K \cdot R^{test}\left(t_l\right)+1}{R^{test}\left(t_l\right)+1}\right\} \\ 
	& =\mathbb{E}\left\{\frac{\left|D^{cal}\right|+K\cdot \left|D^{test}\right|+ K }{\left|D^{test}\right|+1}\right\}.
\end{aligned}$$
Combining the above results, FSR can be controlled:
$$
\begin{aligned}
	\mbox{FSR} & \leq \alpha\cdot \mathbb{E}\left[\frac{\left|D^{cal}\right|+1}{\left|D^{c a l}\right|+K \cdot \left|D^{test}\right|+K} \cdot \frac{\left|D^{cal}\right|+K\cdot \left|D^{test}\right|+K}{\left|D^{test}\right|+1} \cdot \frac{\left|W^{test}\right|}{\left|W^{cal}\right|+1}\right] \\
	& =\alpha \cdot \mathbb{E}\left\{\frac{\left|D^{cal}\right|+1}{\left|W^{cal}\right|+1}\right\} \mathbb{E}\left\{\frac{\left|W^{test}\right|}{\left|D^{test}\right|}\right\} \leq \gamma \alpha.
\end{aligned}
$$
\end{proof}

\begin{proof}[Proof of Theorem \ref{multiclassoptimal}]
	From the discussion above, we know that the problem is equivalent to \begin{equation}\label{FSRcon}
			\min \frac{\sum_{i \in D^{test}}\mathbb{I}(\hat{Y}_{i} = 0)}{n^{test}} , \
			s.t. \ \mathbb{P}(Y = \hat{Y}| \hat{Y} \neq 0) \geq \alpha.
	\end{equation}Where $$\beta =  \frac{(n^{test} - n^{id})\cdot \alpha}{n^{test}},$$ and $n^{id}$ is the number of indecisions of the solution of problem \eqref{FSRcon}. And due to the monotonicity of $\alpha$ and $ n^{id}$, we know that $\alpha$ and $\beta$ are in one-to-one correspondence. From Section \ref{sectionthree}, we know that $S(x)$ can reach the minimum value of $\mbox{ETS}$ when controlling \mbox{mFSR} to be $\alpha$. Then, through the mapping $ F(Y)$, we know that $S(x)$ can reach the minimum value of $||F(\hat{Y})||_{1}$ when controlling $ R$ to be $\beta$.
\end{proof}

\begin{proof}[Proof of Theorem \ref{optimaloracleKclass}]
According to the definition, there is obviously $ P(Y\neq \hat{Y} | S\geq t) = P(Y\neq Y^{\prime} | S\geq t) $. where $$\hat{Y}_{j} = \underset{i}{{\arg\max} \, S_{j}^{i}} \cdot \mathbb{I}(S_{j} \geq t) , \ Y^{\prime}_{j} = \underset{i}{{\arg\max} \, S_{j}^{i}}. $$ for $i = 1,2,...,K $ and $ j = 1,2,...,n^{test} $. The proof is consistent with the binary classification situation. 
\begin{equation}\label{eq_mfsr_K}
	\begin{split}
		\mbox{mFSR}(t) = \frac{\mathbb{E} \{ \sum_{i\in D^{test}}(1-S_{i})\cdot \mathbb{I}(S_{i}\geq t)\}}{\mathbb{E}\{\sum_{i\in D^{test}}\mathbb{I}(S_{i}\geq t)\}},
	\end{split}
\end{equation}
where $ S_{i} = \max\{S^{1}(x_{i}),S^{2}(x_{i}),...,S^{K}(x_{i}) \}$. From the formula above, we have:
$$\alpha \cdot \mathbb{E}\left\{\sum_{i\in D^{test}}\mathbb{I}(S_{i}\geq t)\right\} = \mathbb{E} \left\{ \sum_{i\in D^{test}}(1-S_{i})\cdot \mathbb{I}(S_{i}\geq t)\right\} $$ 
Under the definition of $\delta_{OR}$, we have \begin{equation}
    \mathbb{E} \left\{\sum_{i\in D^{test}}(1-S_{i}-\alpha)\cdot\mathbb{I}(S_{i} \geq t)\right\} = 0,
\end{equation}
\begin{equation}\label{eq3_K}
    \mathbb{E} \left\{\sum_{i\in D^{test}}\sum_{j = 1}^{K}(1-S^{j}_{i}-\alpha)\cdot\mathbb{I}(\delta_{OR}^{i} = j)\right\} = 0,
\end{equation} and have the equation $ \sum_{j = 1}^{K}S^{j} = 1 $. Let $\delta \in \{0,1,2,..,K\}^{n^{test}}$ be a general selection rule in $D_{\alpha}$. Then the \mbox{mFSR} constraints for $\delta$ implies that 
\begin{equation}\label{eq4_K}
	\mathbb{E} \left\{\sum_{i\in D^{test}}\sum_{j = 1}^{K}(1-S^{j}_{i}-\alpha)\cdot\mathbb{I}(\delta^{i} = j)\right\} \leq 0.
\end{equation}
The \mbox{ETS} of $\delta = \delta_{i}: 1 \leq i \leq n^{test}$ is given by
\begin{equation}
	\mbox{ETS} = \mathbb{E} \left\{ \sum_{i \in D^{test}} \sum_{j=1}^{K}S^{j}_{i}\cdot\mathbb{I}(\delta^{i} = j) \right\}.
\end{equation}
According to equations \eqref{eq3_K} and \eqref{eq4_K}, we have :
\begin{equation}\label{eq6_K}
	\mathbb{E}\left\{\sum_{i \in D^{test}} \sum_{j=1}^{K}(1- S^{j}_{i}-\alpha)\cdot(\mathbb{I}(\delta^{i}_{OR} = j)-\mathbb{I}(\delta^{i} = j)) \right\} \geq 0.
\end{equation} 
Let $\lambda_{OR} = (1-Q^{-1}(\alpha)-\alpha)/Q^{-1}(\alpha)$. From equation \eqref{lam0}, we have $ t \leq 1-\alpha$, so $ Q^{-1}(\alpha) \leq 1- \alpha $, $\lambda \geq 0 $. 
So the oracle rule can be written as $$\delta^{i}_{OR} = \underset{j}{{\arg\max} \, S^{j}_{i}} \cdot \mathbb{I}\{\frac{1 - S_{i}-\alpha}{S_{i}}<\lambda_{OR}\},$$ due to the monotonicity of the function$$ f(x) = \frac{x-\alpha}{1- x}.$$

Using the expression and techniques similar to the Neyman-Pearson Lemma, when $\mathbb{I}(\delta^{i}_{OR} = j) = 1$, then $1 - S_{i}^{j} - \alpha -\lambda_{OR}  S_{i}^{j} < 0$, $\mathbb{I}(\delta^{i}_{OR}) - \mathbb{I}(\delta^{i}) \geq 0 $, and vice versa.  So we claim that the following result holds:
$$(\mathbb{I}(\delta^{i}_{OR}= j)-\mathbb{I}(\delta^{i}= j))\cdot(1-S^{j}_{i}-\alpha -\lambda_{OR}  S^{j}_{i}) \leq 0.$$
It follows that:
\begin{equation}\label{eq7_K}
	\mathbb{E}\left\{\sum_{i\in D^{test}} \sum_{j=1}^{K}(\mathbb{I}(\delta^{i}_{OR}= j)-\mathbb{I}(\delta^{i}= j))\cdot(1-S^{j}_{i}-\alpha-\lambda_{OR}S^{j}_{i})\right\} \leq 0.
\end{equation}
According to equation \eqref{eq6_K} and \eqref{eq7_K} we have \begin{equation}
		\lambda_{OR}(\mbox{ETS}_{\delta_{OR}}- \mbox{ETS}_{\delta}) = \lambda_{OR} \sum_{i \in D^{test}} \sum_{j=1}^{K}(\mathbb{I}(\delta^{i}_{OR}= j)-\mathbb{I}(\delta^{i}= j))\cdot S^{j}_{i}\geq 0.
\end{equation}
\end{proof}

\begin{proof}[Proof of Theorem \ref{data_driven_K}]
From the data-driven procedure, we only need to estimate $\frac{V^{test}(\tau)}{V^{cal}(\tau)+1}$
    \begin{equation}
	\begin{split}
		\mbox{FSP}(\tau) &= \frac{V^{test}(\tau)}{V^{cal}(\tau)+1}\cdot \frac{V^{cal}(\tau)+1}{R^{test}(\tau)\bigvee 1} \\
		&= \hat{Q}(\tau)\cdot \frac{n^{cal}+1}{n^{test}} \cdot \frac{V^{test}(\tau)}{V^{cal}(\tau)+1} \\
		&\leq \alpha \cdot \frac{n^{cal}+1}{n^{test}} \cdot \frac{V^{test}(\tau)}{V^{cal}(\tau)+1}.
	\end{split}
\end{equation}
and due to the exchangeability between calibration data and test data, $\frac{V^{test}(s_{k})}{V^{cal}(s_{k})+1}$ is still a backward martingale, where $s_{k}$ corresponds to the threshold(time) when exactly k subjects, combining the subjects in both $ D^{cal}$ and $ D^{test}$, are mistakenly classified, that is whose score of the true class is not the maximum  in $\{S^{1}, S^{2},...,S^{K}\}$, so $ V^{test}(t_{l})$ is $ W^{test}$ and $ V^{cal}(t_{l})$ is $ W^{cal}$. So
\begin{equation}
	\begin{split}
		\mbox{FSR} &= \mathbb{E}(\mbox{FSP}(\tau)) \\
		&\leq \alpha \cdot \mathbb{E}[\frac{n^{cal} + 1}{n^{test}}\cdot  \mathbb{E}\{\frac{V^{test}(\tau)}{V^{cal}(\tau) + 1}| D^{cal},D^{test}\}] \\ 
		&\leq\alpha\cdot\mathbb{E}\{\frac{p_{test}}{p_{cal}}\}.
	\end{split}
\end{equation}
Here, we also control the \mbox{FSR} at $\alpha\cdot\mathbb{E}\{\frac{p_{test}}{p_{cal}}\}$, so we need to ensure that $\mathbb{E}\{\frac{p^{test}}{p^{cal}}\}$ is still 1.

\begin{equation}\begin{split}
    \mathbb{E}(p^{test}) &= 1 - \sum_{i = 1}^{K}\mathbb{P}(\underset{j}{{\arg\max}} \, S^{j}(x) = i, c = i) \\
    & = 1 - \sum_{i = 1}^{K}\mathbb{P}(\underset{j}{{\arg\max}} \, S^{j}(x) = i | c = i) \cdot \mathbb{P}(c = i)\end{split}
\end{equation}
where $x \in D^{test} $ or $D^{cal}$, and \begin{equation}
\begin{split}
        \mathbb{P}(\underset{j}{{\arg\max}} \,  S^{j}(x) = i | c = i) &= \sum_{\sigma\{ 1,2,..i-1,i+1,...,K\}} \int^{1}_{\frac{1}{n}}f_{i}^{i}(s_{i})\cdot\int^{s_{i}}_{\frac{1-s_{i}}{n-1}}f_{i}^{\sigma(1)}(s_{2})\cdot ... \\ & \cdot\int^{s_{K-2}}_{\frac{1-s_{i}-s_{1}-...-s_{K-2}}{2}}f_{i}^{\sigma(K-1)}(s_{K-1}) ds_{K-1}...ds_{1}ds_{i}
\end{split}
\end{equation}
is independent of $\pi$. With the exchangeability between $D^{test}$ and $D^{cal}$, if the rank of function matrix F,
$$
 F = \left\{
 \begin{matrix}
    F^{1}_{1} & F^{1}_{2} & ... & F^{1}_{K}\\  
    F^{2}_{1} & F^{2}_{2} & ... & F^{2}_{K}\\
    ... & ... & ... & ...\\
    F^{K}_{1} & F^{K}_{2} & ... & F^{K}_{K}
  \end{matrix}
  \right\}
$$ is K, then by Cramer's Rule, $p^{test}_{i} = p^{cal}_{i}$ for each $i$, so $\mathbb{E}(p^{test}) = \mathbb{E}(p^{cal})$. The proof is complete.
\end{proof}

\end{appendix}

\end{document}